\documentclass[12pt]{amsart}

\usepackage{amsmath}
\usepackage{amsfonts}
\usepackage{amssymb}
\usepackage{amsthm}
\usepackage{hyperref}
\usepackage{enumerate}
\usepackage{cleveref}
\usepackage{mathrsfs}
\usepackage{lscape}
\usepackage{geometry}
\usepackage{multicol}

\newtheorem{theorem}{Theorem}[section]
\newtheorem{lemma}[theorem]{Lemma}
\newtheorem{proposition}[theorem]{Proposition}

\newtheorem{definition}[theorem]{Definition}
\newtheorem{remark}[theorem]{Remark}
\newtheorem*{remark*}{Remark}
\newtheorem{conjecture}[theorem]{Conjecture}
\Crefname{conjecture}{Conjecture}{Conjectures}

\theoremstyle{plain}

\theoremstyle{plain}

\newcommand{\N}{\mathbb{N}}

\newcommand{\Z}{\mathbb{Z}}
\newcommand\smod[1]{\ \left(\operatorname{mod} #1\right)}

\newcommand{\R}{\mathbb{R}}
\newcommand{\C}{\mathbb{C}}

\newcommand{{\D}}{\delta}

\newcommand{\eps}{\varepsilon}

\newcommand{\SL}{\operatorname{SL}}

\newcommand{\SLZ}{\SL_2(\Z)}
\newcommand{\abcd}{\left(\begin{smallmatrix} a & b \\ c & d \end{smallmatrix}\right)}
\newcommand{\bigabcd}{\begin{pmatrix} a & b \\ c & d \end{pmatrix}}

\newcommand{\calF}{\mathcal{F}}

\newcommand{\tr}{\operatorname{trace}}

\newcommand{\bbM}{\mathbb{M}}
\newcommand{\str}{\operatorname{strace}}
\newcommand{\dodd}{\delta_{odd}}
\newcommand{\calT}{\mathcal{T}}
\newcommand{\mult}{\operatorname{mult}}
\newcommand{\HH}{\mathfrak{H}}
\newcommand{\fa}{\mathfrak{a}}
\newcommand{\lcm}{\operatorname{lcm}}
\newcommand{\pr}{\operatorname{pr}}
\renewcommand{\t}{\mathbf{t}}

\numberwithin{equation}{section}
\numberwithin{table}{section}

\author{Michael J. Griffin and Michael H. Mertens}
\address{Department of Mathematics, Princeton University, Fine Hall, Washington Rd
Princeton, NJ 08544  }
\email{mjgriff1728@gmail.com }
\address{Department of Mathematics and Computer Science, Emory University, 400 Dowman Drive, Atlanta, GA 30322} 
\email{michael.mertens@emory.edu}


\title{A proof of the Thompson Moonshine Conjecture}

\begin{document}
\begin{abstract}
In this paper we prove the existence of an infinite dimensional graded super-module for the finite sporadic Thompson group $Th$ whose McKay-Thompson series are weakly holomorphic modular forms of weight $\frac 12$ satisfying properties conjectured by Harvey and Rayhaun.
\end{abstract}

\maketitle

\section{Introduction and statement of results}

One of the greatest accomplishments of $20^{\rm th}$ century mathematics was certainly the classification of finite simple groups. The study of the representation theory of one of these groups, the \emph{Monster group} $\bbM$, the largest of the 26 sporadic simple groups, revealed an intriguing connection to modular forms: McKay and Thompson \cite{Thompson} were the first to observe that the dimensions of irreducible representations of the Monster group are closely related to Klein's modular invariant
\[J(\tau)=q^{-1}+\sum_{n=1}j_nq^n=q^{-1}+196\,884q+21\,493\,760q^2+864\,299\,970q^3+O(q^4),\]
the Hauptmodul for the full modular group. Here and throughout, $\tau=u+iv$, $u,v\in\R$, is a variable living in the complex upper half-plane $\HH$ and $q:=e^{2\pi i\tau}$. More precisely, the first irreducible representations of $\bbM$ have dimensions \[\chi_1(1)=1,\ \chi_2(1)=196\,883,\ \chi_3(1)=21\,296\,876,\ \chi_4(1)=842\,609\,326,\]
so that one has (for example)
\begin{align*}
j_1&=\chi_1(1)+\chi_2(1)\\
j_2&=\chi_1(1)+\chi_2(1)+\chi_3(1)\\
j_3&=2\chi_1(1)+2\chi_2(1)+\chi_3(1)+\chi_4(1).
\end{align*}
Thompson  further observed that a similar phenomenon occurs when one considers combinations of $\chi_j(g)$ for other $g\in\bbM.$ Based on theses observations, he conjectured in \cite{Thompson2} that there should exist an infinite dimensional graded $\bbM$-module that reflects these combinations. Conway and Norton \cite{CN79} made this more precise (The so-called \emph{Monstrous Moonshine} conjecture), conjecturing that for each conjugacy class of $\bbM$ there is an explicit associated genus $0$ subgroup of $\SL_2(\R)$ whose normalized Hauptmodul coincides with the so-called \emph{McKay--Thompson series} (see \Cref{mckayth} for a definition of this term) of the conjugacy class with respect to the module.
 An abstract proof of this conjecture (i.e. one whose construction of the module is done only implicitly in terms of the McKay--Thompson series) was announced by Atkin, Fong and Smith \cite{Fong,Smith}. Their proof was based on an idea of Thompson. Later, the full conjecture was proven by Borcherds \cite{Borcherds} using a vertex operator algebra previously constructed by Frenkel, Lepowsky, and Meurman \cite{FLM}.

Conway and Norton also observed in \cite{CN79} that Monstrous Moonshine would imply Moonshine phenomena for various subgroups of the Monster. Queen \cite{Queen} computed the Hauptmoduln associated to conjugacy classes of several sporadic groups, among them the Thompson group. Note however, that the Moonshine phenomenon we prove in this paper is not directly related to this generalized moonshine considered by Queen, but more reminiscent of the following. In 2011, Eguchi, Ooguri, and Tachikawa \cite{EOT} observed connections like the ones between the dimensions of irreducible representations of the Monster group and coefficients of the modular function $J$ for the largest Mathieu group $M_{24}$ and a certain weight $\tfrac 12$ \emph{mock theta function}. Cheng, Duncan, and Harvey \cite{CDH1,CDH2} generalized this to Moonshine for groups associated to the 23 Niemeier lattices, the non-isometric even unimodular root lattices in dimension $24$, which has become known as the \emph{Umbral Moonshine Conjecture}. Gannon proved the case of Mathieu moonshine in \cite{Gannon}, the full Umbral Moonshine Conjecture was then proved in \cite{DGO} by Duncan, Ono, and the first author.

In \cite{Harvey}, Harvey and Rayhaun conjecture Moonshine for the \emph{Thompson group} $Th$, a sporadic simple group of order 
\[90\,745\,943\,887\,872\,000= 2^{15} \cdot 3^{10} \cdot 5^3 \cdot 7^2 \cdot 13 \cdot 19 \cdot 31\]
(see \Cref{secSuper} for  definitions and notation).

\begin{conjecture}\label{conj:main}
There exists a graded $Th$-supermodule 
\[W=\bigoplus_{\substack{m=-3 \\ m\equiv 0,1\smod 4}}^\infty W_m,\]
where for $m\geq 0$ the graded component $W_m$ has vanishing odd part if $m\equiv 0\pmod 4$ and vanishing even part if $m\equiv 1\pmod 4$, such that for all $g\in Th$ the McKay-Thompson series
\[\calT_{[g]}(\tau):=\sum_{\substack{m=-3 \\ m\equiv 0,1\smod 4}}^\infty \str_{W_m}(g)q^m\]
is a specifically given weakly holomorphic modular form (see \Cref{secPrelim}) of weight $\tfrac 12$ in the Kohnen plus space.
\end{conjecture}
Here, we prove this conjecture.
\begin{theorem}\label{thm:main}
\Cref{conj:main} is true. Moreover, if $g\in Th$ does not lie in the class $12A$ or $12B$ (see \Cref{except}), the McKay-Thompson series $\calT_{[g]}(\tau)$ is the unique weakly holomorphic modular form in $M_\frac 12^{!,+}(4|g|,\psi_{[g]})$ (see \Cref{secPrelim} for notation) satisfying the following conditions:
\begin{itemize}
\item its Fourier expansion is of the form $2q^{-3}+\chi_2(g)+(\chi_4(g)+\chi_5(g))q^4+O(q^5)$, where $\chi_j$ is the $j^{\rm th}$ irreducible character of $Th$ as given in \Crefrange{char1}{char4}, and all its Fourier coefficients are integers.
\item if $|g|$ is odd, then the only other pole of order $\tfrac 34$ is at the cusp $\frac{1}{2|g|}$, otherwise there is only the pole at $\infty$. It vanishes at all other cusps.
\end{itemize}
If $|g|\neq 36$, it suffices to assume that the Fourier expansion is of the form $2q^{-3}+\chi_2(g)+O(q^4)$.
\end{theorem}

The proof of \Cref{thm:main}, like the proofs of the Mathieu and Umbral Moonshine Conjectures, relies on the following idea. For each $n\geq 0$, the function defined by
\[\omega_n:\ Th\rightarrow \C, g\mapsto \alpha_{[g]}(n),\]
where we write
\[\calT_{[g]}(\tau):=\sum_{n=-3}^\infty \alpha_{[g]}(n)q^n\]
for the tentative McKay-Thompson series conjectured by Harvey and Rayhaun, is a complex valued class function on the Thompson group. Therefore we need to show that $(-1)^n\omega_n$ is a character of $Th$ for every $n$, which is equivalent to the assertion that
\begin{equation}\label{eq:mults}
(-1)^n\omega_n(g)=\sum_{j=1}^{48} m_j(n)\chi_j(g),
\end{equation}
where $m_1,...,m_{48}$ are non-negative integers and $\chi_1,...,\chi_{48}$ are the irreducible characters of $Th$ as defined in \Crefrange{char1}{char4}. Using a variant of Brauer's characterization of generalized characters due to Thompson (see, e.g. \cite[Theorem 1.1]{Smith}), one can reduce this to a finite calculation.

The rest of the paper is organized as follows. In \Cref{secPrelim}, we recall some relevant definitions on supermodules, harmonic Maa{\ss} forms, and the construction of the (tentative) McKay-Thompson series in \cite{Harvey}. In \Cref{secModForms}, we show that these series are in fact all weakly holomorphic modular forms (instead of harmonic weak Maa{\ss} forms) with integer Fourier coefficients and that all the multiplicities $m_j$ in \eqref{eq:mults} are integers. \Cref{secPos} is concerned with the proof of the positivity of these multiplicities, which finishes the proof of \Cref{thm:main}. Finally, in \Cref{secRepl}, we give some interesting observations connecting the McKay-Thompson series to replicable functions.

\section*{Acknowledgements}
The authors would like to thank Ken Ono for instigating this project and John Duncan for many invaluable conversations on the subject and helpful comments on an earlier version of this article.  The first author was supported by the National Science Foundation grant DMS-1502390.

\section{Preliminaries and Notation}\label{secPrelim}
\subsection{Supermodules}\label{secSuper}
We begin by introducing the necessary definitions and notations in \Cref{conj:main}.
\begin{definition}
A vector space $V$ is called a \emph{superspace}, if it is equipped with a $\Z/2\Z$-grading $V=V^{(0)}\oplus V^{(1)}$, where $V^{(0)}$ is called the \emph{even} and $V^{(1)}$ is called the \emph{odd} part of $V$. For an endomorphism $\alpha$ of $V$ respecting this grading, i.e. $\alpha(V^{(i)})\subseteq V^{(i)}$, we define its \emph{supertrace} to be
\[\str(\alpha)=\tr\left(\alpha|_{V^{(0)}}\right)-\tr\left(\alpha|_{V^{(1)}}\right).\]
\end{definition}

Now let $G$ be a finite group and $(V,\rho)$ a representation of $G$. If the $G$-module $V$ admits a decomposition into an even and odd part as above which is compatible with the $G$-action, we call $V$ a \emph{$G$-supermodule}. For a $G$-subsupermodule $W$ of $V$ and $g\in G$ we then write 
\[\str_W(g):=\str\left(\rho|_W(g)\right).\]
Note that $\str(g)$ only depends on the conjugacy class of $g$, which we denote by $[g]$. 

\subsection{Harmonic Maa{\ss} forms}
Harmonic Maa{\ss} forms are an important generalization of classical, elliptic modular forms. In the weight $1/2$ case, they are intimately related to the \emph{mock theta functions}, a term coined by Ramanujan in his famous 1920 deathbed letter to Hardy.  It took until the first decade of the $21^{\rm st}$ century before work by Zwegers \cite{Zwegers}, Bruinier-Funke \cite{BF04} and Bringmann-Ono \cite{BO06,BO10} established the ``right'' framework for these enigmatic functions of Ramanujan's, namely that of harmonic Maa{\ss} forms.
 Since then, there have been many applications of harmonic Maa{\ss} forms both in various fields of pure mathematics, see, for instance, \cite{AGOR, BF, BrO, DIT} among many others, and mathematical physics, especially in regards to quantum black holes and wall crossing \cite{DMZ} as well as Mathieu and Umbral Moonshine \cite{CDH1, CDH2, DGO, Gannon}. For a general overview on the subject we refer the reader to \cite{Ono08,ZagierBourbaki}.

Recall the definition of the congruence subgroup
\[\Gamma_0(N):=\left\{\gamma=\bigabcd\in\SLZ : N|c\right\}\]
of the full modular group $\SLZ$.
\begin{definition}
We call a smooth function $f:\HH\rightarrow\C$ a \emph{harmonic (weak)\footnote{We usually omit the word ``weak'' from now on} Maa{\ss} form} of \emph{weight} $k\in\tfrac 12\Z$ of \emph{level} $N$ with \emph{multiplier system} $\psi$, if the following conditions are satisfied:
\begin{enumerate}
\item We have $f|_k\gamma(\tau)=\psi(\gamma)f(\tau)$ for all $\gamma\in\Gamma_0(N)$ and $\tau\in\HH$, where we define
\[f|_k\gamma (\tau):=\begin{cases} (c\tau+d)^{-k}f\left(\frac{a\tau+b}{c\tau+d}\right) & \text{if }k\in\Z \\
\left(\left(\frac{c}{d}\right)\eps_d\right)^{2k}\left(\sqrt{c\tau+d}\right)^{-2k}f\left(\frac{a\tau+b}{c\tau+d}\right) & \text{if }k\in\frac 12+\Z.
\end{cases}\]
with 
\[\eps_d:=\begin{cases} 1 & d\equiv 1\pmod{4},\\ i & d\equiv 3\pmod 4.\end{cases}\]
and where we assume $4|N$ if $k\notin\Z$.
\item The function $f$ is annihilated by the \emph{weight $k$ hyperbolic Laplacian},
\[\Delta_k f:=\left[-v^2\left(\frac{\partial^2}{\partial u^2}+\frac{\partial^2}{\partial v^2}\right)+ikv\left(\frac{\partial}{\partial u}+i\frac{\partial}{\partial v}\right)\right] f\equiv 0.\]
\item There is a polynomial $P(q^{-1})$ such that $f(\tau)-P(e^{-2\pi i\tau})=O(e^{-cv})$ for some $c>0$ as $v\to\infty$. Analogous conditions are required at all cusps of $\Gamma_0(N)$.
\end{enumerate}
We denote the space of harmonic Maa{\ss} forms of weight $k$, level $N$ and multiplier $\psi$ is denoted by $H_k(N,\psi)$, where we omit the multiplier if it is trivial.
\end{definition}
\begin{remark}\label{remark}
\begin{enumerate}
\item Obviously, the weight $k$ hyperbolic Laplacian annihilates holomorphic functions, so that the space $H_k(N,\psi)$ contains the spaces $S_k(N,\psi)$ of cusp forms (holomorphic modular forms vanishing at all cusps), $M_k(N,\psi)$ of holomorphic modular forms, and $M_k^!(N,\psi)$ of weakly holomorphic modular forms (holomorphic functions on $\HH$ transforming like modular forms with possible poles at cusps).
\item It should be pointed out that the definition of modular forms resp. harmonic Maa{\ss} forms with multiplier is slightly different in \cite{Harvey}, where the multiplier is included into the definition of the slash operator $f|_k\gamma$, so that multipliers here are always the inverse of the multipliers there.
\end{enumerate}
\end{remark}
It is not hard to see from the definition that harmonic Maa{\ss} forms naturally split into a holomorphic part and a non-holomorphic part (see for example equations (3.2a) and (3.2b) in \cite{BF04}).
\begin{lemma}\label{lem:split}
Let $f\in H_{k}(N,\psi)$ be a harmonic Maa{\ss} form of weight $k\neq 1$ such that $\psi(\left(\begin{smallmatrix} 1 & 1 \\ 0 & 1\end{smallmatrix}\right))=1$. Then there is a canonical splitting
\begin{equation}\label{eq:split}
f(\tau)=f^+(\tau)+f^-(\tau),
\end{equation}
where for some $m_0\in\Z$ we have the Fourier expansions
\[f^+(\tau):=\sum\limits_{n=m_0}^\infty c_f^+(n)q^n,\]
and
\[f^-(\tau):=\sum\limits_{\substack{n=1}}^\infty \overline{c_f^-(n)}n^{k-1}\Gamma(1-k;4\pi nv)q^{-n},\]
where $\Gamma(\alpha;x)$ denotes the usual incomplete Gamma-function.
\end{lemma}

In the theory of harmonic Maa{\ss} forms, there is a very important differential operator that associates a weakly holomorphic modular form to a harmonic Maa{\ss} form \cite[Proposition 3.2 and Theorem 3.7]{BF04}, often referred to as its \emph{shadow}\footnote{In the literature, the shadow is often rather associated to the holomorphic part $f^+$ of a harmonic Maa{\ss} form $f$ rather than to $f$ itself.}.
\begin{proposition}
The operator
\[\xi_{k}:H_{k}(N,\psi)\rightarrow S_{2-k}\left(N,\overline{\psi}\right),\:f\mapsto\xi_{k}f:=2iv^{k}\overline{\frac{\partial f}{\partial\overline{\tau}}}\]
is well-defined and surjective with kernel $M_{k}^!(N,\nu)$. Moreover, we have that
\[(\xi_{k}f)(\tau)=-(4\pi)^{1-k}\sum\limits_{n=1}^\infty c_f^-(n)q^n\]
and we call this cusp form the \emph{shadow} of $f$.
\end{proposition}

The $\xi$-operator can also be used to define the \emph{Bruinier-Funke pairing}
\[\{\cdot,\cdot\}: S_{2-k}(N,\overline{\psi})\times H_k(N,\psi),\quad (g,f)\mapsto \{g, f\}:=\langle g, \xi_k(f)\rangle,\]
where $\langle\cdot,\cdot\rangle$ denotes the Petersson inner product on the space of cusp forms. 

We will later use the following result \cite[Proposition 3.5]{BF04}.
\begin{proposition}\label{BFpairing}
For $f=f^++f^-$ as in \Cref{lem:split} and $g=\sum_{n=1}^\infty b_nq^n\in S_{2-k}(N,\overline{\psi})$ such that $f$ grows exponentially only at the cusp $\infty$ and is bounded at all other cusps of $\Gamma_0(N)$, we have that
\[\{g,f\}=\sum_{n<0} c_f^+(n)b(-n).\]
If $f$ has poles at other cusps, the pairing is given by summing the corresponding terms using the $q$-series expansions for $f$ and $g$ at each such cusp.
\end{proposition}
\subsection{Rademacher sums and McKay-Thompson series}\label{mckayth}
Here we recall a few basic facts about Poincar\'e series, Rademacher sums, and Rademacher series. For further details, the reader is referred to \cite{CD11,CD12,DF} and the references therein.

An important way to construct modular forms of a given weight and multiplier is through Poincar\'e series. If one assumes absolute and locally uniform convergence, then the function
\[P^{[\mu]}_{\psi,k}(\tau):=\sum_{\gamma\in\Gamma_\infty\setminus \Gamma_0(N)} \overline{\psi}(\gamma)q^\mu|_k\gamma,\]
where $\Gamma_\infty=\left\langle \pm T\right\rangle$ with $T=\left(\begin{smallmatrix} 1 & 1 \\ 0 & 1 \end{smallmatrix}\right)$ denotes the stabilizer of the cusp $\infty$ in $\Gamma_0(N)$ and $\mu\in \frac{\log(\psi(T))}{2\pi i}+\Z$, transforms like a modular form of weight $k$ with multiplier $\psi$ under the action of $\Gamma_0(N)$ and is holomorphic on $\HH$. In fact it is known that we have absolute and locally uniform convergence for weights $k>2$ and in those cases, $P^{[\mu]}_{\psi,k}$ is a weakly holomorphic modular form, which is holomorphic if $\mu\geq 0$ and cuspidal if $\mu>0$.

For certain groups and multiplier systems, one can obtain conditionally, locally uniformly convergent series, now called \emph{Rademacher sums}, for weights $k\geq 1$, by fixing the order of summation as follows. Let for a positive integer $K$
\[\Gamma_{K,K^2}(N):=\left\{\abcd\in\Gamma_0(N)\: : \: |c|<K\text{ and }|d|<K^2\right\}.\] 
One can then define the Rademacher sum
\[R^{[\mu]}_{\psi,k}(\tau)=\lim\limits_{K\to \infty} \sum_{\gamma\in\Gamma_\infty\setminus \Gamma_{K,K^2}(N)} \overline{\psi}(\gamma)q^\mu|_k\gamma.\]
Rademacher observed \cite{Rademacher} that one can in addition regularize the summands individually to obtain convergence for weights $k<1$, see e.g. equations (2.26) and (2.27) in \cite{CD12}. He originally applied this to obtain an exact formula for the coefficients of the modular $j$-function.

In this paper, we especially need to look at Rademacher sums of weights $\tfrac 12$ and $\tfrac 32$ for $\Gamma_0(4N)$ with multiplier
\begin{equation}\label{psi}
\psi_{N,v,h}\left(\gamma\right):=\exp\left(-2\pi iv\frac{cd}{Nh}\right),
\end{equation}
where $\gamma=\abcd\in\Gamma_0(4N)$ and $v,h$ are integers with $h|\gcd(4N,96)$. First we establish convergence of these series.
\begin{proposition}\label{Rademacherconv}
For a positive integer $N$ and multiplier $\psi=\psi_{N,v,h}$ as in \eqref{psi}, the Rademacher sums $R^{[-3]}_{\psi,\frac 12}(\tau)$ and $R^{[3]}_{\overline\psi,\frac 32}$ converge locally uniformly on $\HH$ and therefore define holomorphic functions on $\HH$.
\end{proposition}
\begin{proof}
By following the steps outlined in \cite[Section 5]{CD11} to establish the convergence of weight $\tfrac 12$ Rademacher series with a slightly different multiplier system (related to that of the Dedekind eta function) \emph{mutatis mutandis}, we find that the Rademacher sums we are interested in converge, assuming the convergence at $s=\tfrac 34$ of the Kloosterman zeta function 
\[\sum_{c=1}^\infty\frac{K_\psi(-3,n,4Nc)}{(4Nc)^{2s}}\]
with 
\begin{equation}\label{Kloost}
K_\psi(m,n,c):=\sideset{}{^*}\sum_{d\smod{c}}\psi(c,d)\left(\frac cd\right)\eps_d e\left(\frac{m\overline{d}+nd}{c}\right),
\end{equation}
where the ${}^*$ at the sum indicates that it runs over primitive residue classes modulo $c$, $\overline{d}$ denotes the multiplicative inverse of $d$ modulo $c$, and $e(\alpha):=\exp(2\pi i\alpha)$ as usual. We omit the subscript if $\psi=1$. In order to establish positivity of the multiplicities of irreducible characters in \Cref{secPos}, we will show not only convergence of this series, but even explicit estimates for its value, which will complete the proof.
\end{proof}
Since the Rademacher sum $R^{[-3]}_{\psi,\frac 12}(\tau)$ is $1$-periodic by construction, it has a Fourier expansion, which can (at least formally) be established by standard methods. Projecting this function to the Kohnen plus space then yields the function
\begin{equation}\label{Z}
Z_{N,\psi}(\tau):=q^{-3}+\sum_{\substack{n=0 \\ n\equiv 0,1\smod{4}}}^\infty A_{N,\psi}(n)q^n,
\end{equation}
where $A_{N,\psi}$ is given by
\begin{equation}\label{A}
\begin{aligned}
A_{N,\psi}(0)&:=\frac{\pi\sqrt{3}}{2 N^\frac 32}(1-i)\sum_{c=1}^\infty \left(1+\dodd(Nc)\right)\frac{K_\psi(-3,0,4Nc)}{(4Nc)^\frac 32},\\
A_{N,\psi}(n)&:=\frac{\pi\sqrt{2}}{4N}\left(\frac 3n\right)^\frac 14(1-i)\sum_{c=1}^\infty(1+\dodd(Nc))\frac{K_\psi(-3,n,4Nc)}{4Nc}I_\frac 12\left(\frac{\pi\sqrt{3n}}{Nc}\right).
\end{aligned}
\end{equation}
Here, $I_\frac 12$ denotes the usual modified Bessel function of the first kind of order $\tfrac 12$ and
\[\dodd(k):=\begin{cases} 1 & k\text{ odd,} \\ 0 & k\text{ even.}\end{cases}\]

For each conjugacy class $[g]$ of the Thompson group $Th$, we associate integers $v_g$ and $h_g$ (where $h_g|96$) as specified in \Cref{mult} and the character $\psi_{[g]}:=\psi_{|g|,v_g,h_g}$, where $|g|$ denotes the order of $g$ in $Th$, as well as a finite sequence of rational numbers $\kappa_{m,g}$ which are also given in \Cref{mult} and define the function
\begin{equation}\label{calF}
\calF_{[g]}(\tau):=2Z_{|g|,\psi_{[g]}}(\tau)+\sum_{\substack{m>0 \\ m^2\mid h_g|g|}}\kappa_{m,g}\vartheta(m^2\tau),
\end{equation}
with
\[\vartheta(\tau):=\sum_{n\in\Z}q^{n^2}.\]
This is going to be the explicitly given weakly holomorphic modular form (see \Cref{prop:weak}) mentioned in \Cref{conj:main}, meaning that we have
\[\calT_{[g]}(\tau)=\calF_{[g]}(\tau)\]
for all conjugacy classes $[g]$ of $Th$.

We now prove and recall some important facts about Rademacher sums that we shall use later on. As in \cite[Propositions 7.1 and 7.2]{CD11}, one sees the following.
\begin{proposition}\label{Rademachershadow}
The Rademacher sum $R^{[-3]}_{\psi,\frac 12}(\tau)$ with $\psi$ as in \eqref{psi} is a mock modular form of weight $\tfrac 12$ whose shadow is a cusp form with the conjugate multiplier $\overline\psi$, which is a constant multiple of the Rademacher sum $R^{[3]}_{\psi,\frac 32}$.
\end{proposition}

Next we establish the behaviour of Rademacher sums at cusps. Here we have to take into account that the sums we look at are projected into the Kohnen plus space which might affect the behaviour at cusps. For a function $f\in M_{k+\frac 12}^!(\Gamma_0(4N))$, where $k$ is an integer and $N$ is odd, the projection of $f$ to the plus space is defined by
\[f|\pr=(-1)^{\lfloor\frac{k+1}{2}\rfloor}\frac{1}{3\sqrt{2}}\sum_{v=-1}^2( f|B\cdot A_v)+\frac 13 f,\]
where 
\[B=\begin{pmatrix}
4 & 1 \\ 0 & 4
\end{pmatrix}\quad\text{and}\quad A_v=\begin{pmatrix}
1 & 0 \\ 4Nv & 1
\end{pmatrix}.\]

Using this projection operator, one sees that the following is true. The proof is similar in nature to that of Proposition 3 in \cite{Kohnen} and is carried out in some detail (for the special case where $N=p$ is an odd prime) in \cite[Section 2]{GreenJenkins}. 
\begin{lemma}\label{plusspace}
Let $N$ be odd and $f\in H_{k+\frac 12}(\Gamma_0(4N))$ for some $k\in\N_0$, such that 
\[f^+(\tau)=q^{-m}+\sum_{n=0}^\infty a_nq^n\]
for some $m>0$ with $-m\equiv 0,(-1)^k\pmod 4$ has a non-vanishing principal part only at the cusp $\infty$ and is bounded at the other cusps of $\Gamma_0(4N)$. Then the projection $f|\pr$ of $f$ to the plus space has a pole of order $m$ at $\infty$ and has a pole of order $\frac m4$ either at the cusp $\frac 1N$ if $m\equiv 0\pmod 4$ or at the cusp $\frac1{2N}$ if $-m\equiv(-1)^k\pmod 4$ and is bounded at all other cusps.
\end{lemma}
\begin{proof}
In order to compute the expansion of $f|\pr$ at a given cusp $\fa=\frac{a}{c}$, we compute
\[f|\pr|\sigma_\fa,\]
where $\sigma_\fa=\abcd\in\SLZ$. Multiplying out the matrices, we see that this is (up to a constant factor) equal to
\begin{equation}\label{projcusp}
(-1)^{\lfloor\frac{k+1}{2}\rfloor}\frac{1}{3\sqrt{2}}\sum_{v=-1}^2 f|\begin{pmatrix}
(4+4Nv)a+c & (4+4Nv)b+d \\ 16Nva+4c & 16Nvb+4d
\end{pmatrix}+\frac 13 f|\sigma_\fa.
\end{equation}
By assumption, this function can only have a pole at $\infty$ if the denominator of the fraction $\frac{(4+4Nv)a+c}{16Nva+4c}$ in lowest terms (where we allow the denominator to be $0$ which we interpret as $\infty$) is divisible by $4N$, which is easily seen to imply that $N|c$. Since there are only three inequivalent cusps of $\Gamma_0(4N)$ whose denominator is divisible by $N$, represented by $\infty$, $\frac 1N$, $\frac 1{2N}$, we can restrict ourselves to 
\[\sigma_{\frac 1N}=\begin{pmatrix}
1 & 0 \\ N & 1
\end{pmatrix}\quad\text{and}\quad \sigma_{\frac 1{2N}}=\begin{pmatrix}
1 & 0 \\ 2N & 1
\end{pmatrix}.\]
Plugging first $\sigma_\frac 1N$ into \eqref{projcusp} we find for each $v=-1...2$ that 
\[f|B\cdot A\cdot \sigma_\frac 1N=f|\begin{pmatrix}
4+4Nv+N & \beta_v \\ 16Nv+4N & \delta_v
\end{pmatrix}\cdot\begin{pmatrix}
1 & \delta_v-4\beta_v \\ 0 & 16
\end{pmatrix}=Cf\left(\frac{\tau+\delta_v-4\beta_v}{16}\right),\]
where $\beta_v,\delta_v\in\Z$ such that 
\[\det \begin{pmatrix}
4+4Nv+N & \beta_v \\ 16Nv+4N & \delta_v
\end{pmatrix}=1\]
and $C\in\C$ is a constant that a priori depends on $v$, but by working out the corresponding automorphy factors, one sees with easy, elementary methods that it does indeed not. Note that it does however depend on $N$. Furthermore it is not hard to see that the difference $\delta_v-4\beta_v$ runs through all residue classes modulo $16$ that are congruent to $N$ modulo $4$. This implies that in 
\[\sum_{v=-1}^2 f\left(\frac{\tau+\delta_v-4\beta_v}{16}\right)\]
only powers of $q^\frac{1}{16}$ survive whose exponent is divisible by $4$. Since by assumption we have that $f|\sigma_\frac 1N=O(1)$ as $\tau\to\infty$, we therefore see that $f|\pr$ has a pole of order $\frac m4$ at the cusp $\tfrac 1N$ if and only if $m$ is divisible by $4$. 

For the cusp $\tfrac 1{2N}$, the argumentation is analogous. One finds that
\[f|B\cdot A\cdot \sigma_\frac 1{2N}=f|\begin{pmatrix}
2+2Nv+N & \beta_v' \\ 8Nv+4N & \delta_v'
\end{pmatrix}\cdot\begin{pmatrix}
2 & \delta_v-4\beta_v \\ 0 & 8
\end{pmatrix}=C_v'f\left(\frac{2\tau+\delta_v-4\beta_v}{8}\right),\]
where $\beta_v',\delta_v',C_v'$ have the analogous meaning as $\beta_v,\delta_v,C$ above, with the only difference that $C_v'$ actually does depend on $v$. The dependence on $v$ is so that in the summation only powers of $q^\frac 14$ with exponents $\equiv (-1)^k\pmod 4$ survive, which implies our Lemma.
\end{proof}

For even $N$, it turns out that the Rademacher series are automatically in the plus space. This follows immediately from the next lemma.

\begin{lemma}\label{lemKloost}
Let $m, n\in\Z$ such that $m\not\equiv n\pmod 4$ and $c\in\N$ be divisible by $8$. Then we have
\[K(m,n,c)=0\]
with $K(m,n,c)$ as defined in \eqref{Kloost}.
\end{lemma}
\begin{proof}
We write $c=2^\ell c'$ with $\ell\geq 3$ and $c'$ odd. By the Chinese Remainder Theorem one easily sees the following multiplicative property of the Kloosterman sum,
\begin{equation}
K(m,n,c)=K(m\overline{c'},n\overline{c'},2^\ell)\cdot S(m\overline{2^\ell},n\overline{2^\ell},c'),
\end{equation}
where
\[S_c(m,n,c):=\sideset{}{^*}\sum_{d\smod{c}} \left(\frac cd\right)e\left(\frac{m\overline{d}+nd}{c}\right)\]
is a Sali\'e sum and $\overline{c'}$ denotes the inverse of $c'$ modulo $2^\ell$ and $\overline{2^\ell}$ denotes the inverse of $2^\ell$ modulo $c'$.

Therefore, it suffices to show the lemma for $c=2^\ell$ with $\ell\geq 3$. The case where $\ell=3$ can be checked directly, so assume $\ell\geq 4$ from now on. In this case, it is straightforward to see that
\[\left(\frac {2^\ell}{d}\right)=\left(\frac{2^\ell}{d+2^{\ell-1}}\right),\qquad \eps_d=\eps_{d+2^{\ell-1}},\qquad \text{and}\qquad \overline{d+2^{\ell-1}}=\overline{d}+2^{\ell-1}.\]
This yields that for $m\not\equiv n\pmod 2$ we have that
\[\left(\frac{2^\ell}{d+2^{\ell-1}}\right)\eps_{d+2^{\ell-1}}\cdot e\left(\frac{m(\overline{d+2^{\ell-1}})+n(d+2^{\ell-1})}{2^\ell}\right)=-\left(\frac{2^\ell}{d}\right)\eps_d e\left(\frac{m\overline{d}+nd}{2^\ell}\right)\]
for all odd $d\in\{1,...,2^{\ell-1}-1\}$, so that the summands in the Kloosterman sum pair up with opposite signs, making the sum $0$ as claimed. 

If $m$ and $n$ have the same parity, but are not congruent modulo $4$, a similar pairing also works. In this case we find through similar reasoning that for $\ell\geq 5$ we have
\begin{align*}
\left(\frac{2^\ell}{d}\right)\eps_d e\left(\frac{m\overline{d}+nd}{2^\ell}\right)&=\left(\frac{2^\ell}{d+2^{\ell-1}}\right)\eps_{d+2^{\ell-1}}\cdot e\left(\frac{m(\overline{d+2^{\ell-1}})+n(d+2^{\ell-1})}{2^\ell}\right)\\
&=-\left(\frac{2^\ell}{d+2^{\ell-2}}\right)\eps_{d+2^{\ell-2}}\cdot e\left(\frac{m(\overline{d+2^{\ell-2}})+n(d+2^{\ell-2})}{2^\ell}\right).
\end{align*}
Again, we can pair summands with opposite signs, proving the lemma.
\end{proof}
From the preceding two lemmas we immediately find that the following is true.
\begin{proposition}\label{Rademachercusp}
For any $g\in Th$, the function $Z_{|g|,\psi_{[g]}}$ is a mock modular form which has a pole of order $3$ at $\infty$, a pole of order $\tfrac 34$ at $\frac 1{2N}$ if $N$ is odd, and vanishes at all other cusps.
\end{proposition}
\begin{proof}
As described in Appendix E of \cite{CD11} we see that the Rademacher sums $R^{[-3]}_{\psi_{[g]},\frac 12}(\tau)$ have only a pole of order $3$ at $\infty$ and grow at most polynomially at all other cusps. By \Cref{plusspace,lemKloost} we see that the poles are as described in the proposition. The vanishing at all remaining cusps follows as in \cite[Theorem 3.3]{BJO}.
\end{proof}
\section{Identifying the McKay-Thompson series as modular forms}\label{secModForms}
In this section we want to establish that the multiplicities of each irreducible character are integers. To this end, we first establish the exact modularity and integrality properties of the conjectured McKay-Thompson series $\calF_{[g]}(\tau)$, which are stated without proof in \cite{Harvey}.
\begin{proposition}\label{prop:weak}
For each element $g$ of the Thompson group, the function $Z_{|g|,\psi_{[g]}}(\tau)$ as defined in \eqref{Z} lies in the space $M_{\frac 12}^{+,!}(4|g|,\psi_{[g]})\leq M_{\frac 12}^{+,!}(N_{[g]})$ with $N_{[g]}$ as in \Cref{mult}.
\end{proposition}
\begin{proof}
As we know from \Cref{Rademachershadow}, we have
\[\xi_{\frac 12}Z_{|g|,\psi_{[g]}}\in S^+_\frac 32(4|g|,\overline{\psi_{[g]}})\leq S_\frac 32^+(N_{[g]}).\]
The space $S_\frac 32^+(N_{[g]})$ now turns out to be zero-dimensional for 
\begin{align*}
[g]\in\{& 1A, 2A, 3A, 3B, 3C, 4A, 4B, 5A, 6A, 6B, 6C, 7A, 8A, 9A, 9B, 10A, 12A, 12B, 12C, \\
        &13A, 18A, 36A, 36B, 36C \},
\end{align*}
which is directly verifiable using the built-in functions for spaces of modular forms in for example \textsc{Magma} \cite{Magma}. Furthermore, we have the Bruinier-Funke pairing (see \Cref{BFpairing}) combined with \Cref{Rademachercusp} which tell us that
\begin{equation}\label{eqpairing}
\{\xi_\frac 12 Z_{|g|,\psi_{[g]}},Z_{|g|,\psi_{[g]}}\}\overset{.}=c_{Z_{|g|,\psi_{[g]}}}^-(3),
\end{equation}
where the dot above the equal sign indicated an omitted multiplicative (non-zero) constant. More precisely, we can apply \Cref{BFpairing} to the Rademacher sum $R_{\psi,\frac 12}^{[-3]}$ and then use the same reasoning as in the proof of \Cref{plusspace} to see that projection to the plus space only alters this value by a multiplicative non-zero constant, since the additional pole at the cusp $\frac 1{2|g|}$ (if $|g|$ is odd) is directly forced by the plus space condition. 

From \eqref{eqpairing} we can now deduce, because the Petersson inner product is positive definite on the space of cusp forms, that the shadow of $Z_{|g|,\psi_{[g]}}$ must be $0$ if every $f\in S_\frac 32^+(N_{[g]})$ is $O(q^4)$. Again, this can be checked using built-in features of \textsc{Magma}, therefore showing the claim for 
\[[g]\in\{14A, 19A, 20A, 28A, 31A, 31B\}.\]

For the remaining 18 conjugacy classes, one can use the same arguments as above, but with the refinement that instead of looking at the full space $S_\frac 32^+(N_{[g]})$, one looks at the (usually) smaller space $S_\frac 32^+(4|g|,\overline{\psi_{[g]}})$. Since computing bases for these spaces is not something that a standard computer algebra system can do without any further work, we describe how to go about doing this. Let $f\in S_\frac 32^+(4|g|,\overline{\psi_{[g]}})$ for some conjugacy class $[g]$. Then $f\cdot \vartheta$ is a modular form of weight 2 with the same multiplier and level (respectively trivial multiplier and level $N_{[g]}$). Using programs\footnote{available at \url{http://users.wfu.edu/rouseja/eta/}} written by Rouse and Webb \cite{RouseWebb} one can verify that the algebra of modular forms of level $N_{[g]}$, so in particular the space $M_2(N_{[g]})$, is generated by eta quotients. One can also compute a generating system consisting of eta quotients for all remaining $N_{[g]}$ that still need to be considered. Since one can actually compute Fourier expansions of expressions like
\[(f|\gamma)(\tau)\]
for $f(\tau)=\prod_{\delta|N}\eta(\delta \tau)^{r_\delta}$ an eta quotient and $\gamma\in\SLZ$ explicitly, see e.g. \cite[Proposition 2.1]{eta} it is straight-forward linear algebra to compute a basis of the space $M_2(4|g|,\overline{\psi_{[g]}})$ and from there go down to $S_\frac 32^+(4|g|,\overline{\psi_{[g]}})$. A \textsc{Magma} script computing dimensions and bases of these spaces can be obtained from the second author's homepage. Using this script we find that $\dim S^+_\frac 32(4|g|,\overline{\psi_{[g]}})=0$ for 
\[[g]\in\{8B,9C,12D,15A,15B,24A,24B,24C,24D,27A,27B,27C\}\]
and for all remaining conjugacy classes $[g]$, we find that every $f\in S^+_\frac 32(4|g|,\overline{\psi_{[g]}})$ is $O(q^4)$. This completes the proof.
\end{proof}
\begin{proposition}\label{prop:int}
For each $g\in Th$, the functions $\calF_{[g]}(\tau)=\sum_{n=-3}^\infty c_{[g]}(n)q^n$ as defined in \eqref{calF} are all weakly holomorphic modular forms of weight $\tfrac 12$ for the group $\Gamma_0(N_{[g]})$ in Kohnen's plus space with integer Fourier coefficients at $\infty$.
\end{proposition}
\begin{proof}
We have established in \Cref{prop:weak} that the Rademacher series $Z_{|g|,\psi_{[g]}}(\tau)$ are all weakly holomorphic modular forms of weight $\tfrac 12$ for $\Gamma_0(N_{[g]})$ in the plus space. The given theta corrections are holomorphic modular forms of the same weight and level, hence so is their sum. Furthermore, theta functions don't have poles, so that all poles of $\calF_{[g]}$ come from the Rademacher series which has a pole of order $3$ only at the cusps of $\Gamma_0(N_{[g]})$ lying above the cusp $\infty$ on the modular curve $X_0(4|g|)$. Hence the function
\[\calF_{[g]}(\tau)\cdot G(\tau),\]
where $G(\tau)=q^3+O(q^4)\in S^+_{2k-\frac 12}(\Gamma_0(4|g|)$ is a cusp form with integer Fourier coefficients, is a weight $2k$ holomorphic modular form with trivial multiplier under the group $\Gamma_0(N_{[g]})$. Now by the choice of $G(\tau)$ we have that $\calF_{[g]}(\tau)$ has integer Fourier coefficients if and only if $\calF_{[g]}(\tau)\cdot G(\tau)$ has integer Fourier coefficients, which is the case if and only if the Fourier coefficients of this modular form are integers up to the Sturm bound \cite{Sturm} 
\[\frac k6[\SLZ:\Gamma_0(N_{[g]})].\]
In order to compute the necessary Fourier coefficients exactly without relying on the rather slow convergence of the Fourier coefficients of the Rademacher series, one can construct linear combinations of weight $\tfrac 12$ weakly holomorphic eta quotients again using the programs\footnote{available at \url{http://users.wfu.edu/rouseja/eta/}} written by Rouse and Webb \cite{RouseWebb} which have the same principal part at $\infty$ (and the related cusp $\frac 1{2N}$) as the Rademacher series (see \Cref{Rademachercusp}) and the same constant terms as the theta corrections if there are any, wherefore their difference must be a weight $\tfrac 12$ holomorphic cusp form for $\Gamma_0(N_{[g]})$ with trivial multiplier, which by the Serre-Stark basis theorem \cite{SerreStark} is easily seen to be $0$\footnote{A list of the eta quotients and linear combinations are available from the second author's homepage}. The largest bound up to which coefficients need to be checked turns out to be $384$ for $[g]=24CD$.
\end{proof}
\begin{remark}\label{except}
As it turns out, the theta correction 
\[f(\tau)=-\vartheta(4\tau)+3\vartheta(36\tau)\]
$[g]=12AB$ that Harvey and Rayhaun \cite[Table 5]{Harvey} give transforms with a different multiplier than the Rademacher series $Z_{12,\psi_{12AB}}(\tau)$: As one computes directly from the fact that
\[-\vartheta(\tau)+3\vartheta(9\tau)\]
transforms with the multiplier $\psi_{3,1,3}$ under the group $\Gamma_0(12)$, $f(\tau)$ transforms under the group $\Gamma_0(48)$ with the multiplier $\psi_{12,1,3}$, while $Z_{12,\psi_{12AB}}(\tau)$ transforms with the multiplier $\psi_{12,7,12}$. But since both multipliers become trivial on the group $\Gamma_0(144)$, the proposition remains valid.
\end{remark}

We can now establish the uniqueness claim in \Cref{thm:main} very easily.
\begin{proposition}
For $[g]\notin\{12A,12B,36A,36B,36C\}$, we have that the function $\calF_{[g]}(\tau)\in M_\frac 12^{!,+}(4|g|,\psi_{[g]})$ is the unique function satisfying the conditions given in \Cref{thm:main}:
\begin{itemize}
\item its Fourier expansion is of the form $2q^{-3}+\chi_2(g)+O(q^4)$ and all its Fourier coefficients are integers.
\item if $|g|$ is odd, then the only other pole of order $\tfrac 34$ is at the cusp $\frac{1}{2|g|}$, otherwise there is only the pole at $\infty$. It vanishes at all other cusps.
\end{itemize}
For $|g|=36$, $\calF_{[g]}(\tau)$ is uniquely determined by additionally fixing the coefficient of $q^4$ to be $\chi_4(g)+\chi_5(g)$.
\end{proposition}
\begin{proof}
As we have used already, the function $2Z_{|g|,\psi_{[g]}}(\tau)$ has the right behaviour at the cusps so that $\calF_{[g]}(\tau)-2Z_{|g|,\psi_{[g]}}(\tau)$ is a holomorphic weight $\frac 12$ modular form. As it turns out, in all cases but the one where $|g|=36$, this space is at most two-dimensional, which can be seen by the Serre-Stark basis theorem if $\psi_{[g]}$ is trivial or through a computation similar to the one described in the proof of \Cref{prop:weak} if the multiplier is not trivial. Hence prescribing the constant and first term in the Fourier expansion determines the form uniquely. If $|g|=36$, the space of weight $\frac 12$ modular forms turns out to be $3$-dimensional, so that fixing one further Fourier coefficient suffices to determine the form uniquely. 
\end{proof}

We ultimately want to study the multiplicities of the irreducible characters of $Th$. To this end, we now consider the functions
\[\calF_{\chi_j}(\tau):=\frac{1}{|Th|}\sum_{g\in Th} \overline{\chi_j(g)}\calF_{[g]}=\sum_{n=-3}^\infty m_j(n)q^n\]
with $m_j(n)$ as in \eqref{eq:mults}, the generating functions of the multiplicities. We want to show that all those numbers $m_j(n)$ are integers. A natural approach for this would be to view $\calF_{\chi_j}$ as a weakly holomorphic modular form of weight $\tfrac 12$ and level
\[N_{\chi_j}:=\lcm\{N_{[g]} \ :\ \chi_j(g)\ne 0\}\]
and then use a Sturm bound type argument as in the proof of \Cref{prop:int}. However, these levels turn out to be infeasibly large in most cases. For example we have that
\[N_{\chi_1}=2\,778\,572\,160,\]
so one would have to compute at least a few 100 million Fourier coefficients of $\calF_{\chi_1}$ to make such an argument work, which is entirely infeasible. 

This bound can be reduced substantially however by breaking the problem into many smaller problems involving simpler congruences, each of which requires far fewer coefficients to prove. 

We proceed by a linear algebra argument. Let $\mathbf C$ be the coefficient matrix containing the coefficients of the alleged McKay--Thompson series for each conjugacy class. In theory we have that $\mathbf C$ is a $48\times\infty$ matrix. In practice, we take $\mathbf C$ to be a $48\times B$ matrix with $B$  large. Let $\mathbf X$ be the $48\times 48$ matrix with columns indexed by conjugacy classes of $Th$ and rows indexed by irreducible characters, whose $(\chi_i,[g])$-th entry is 
\[
\mathbf X_{(\chi_i,[g])}=\chi_i(g)\cdot\frac{|[g]|}{|Th|}.
\]
Using the first Schur orthogonality relation for characters,
\[\sum_{g\in Th}\chi_i(g)\chi_j(g)=\begin{cases}|Th| &\text{ if } i=j\\
0& \text{ otherwise,}\end{cases}\] 
we have that the rows of the matrix   
$\mathbf m:=\mathbf X \mathbf C$ which are indexed by the characters $\chi_i$ are exactly the multiplicity values under consideration for the given character.

The matrix $\mathbf C$ does not have full rank. Besides the duplicated series (such as $\mathcal T_{[12A]}=\mathcal T_{[12B]}$), we have additional linear relations given in \Cref{TblLinearRelations}. As some of these relations involve the theta functions used as correction terms in the construction, let us define $\mathbf C^+$ to be the matrix extending $\mathbf C$ to include the coefficients of the theta functions  $\vartheta(n^2\tau)$ , for $n=1,2,3,6,9.$ Then there are matrices $\mathbf N,\mathbf N^*$ of dimensions $48\times 35$ and $35\times 53$ respectively, so that $\mathbf N^* \mathbf C^+$ 
has full rank and 
\[
\mathbf m=\mathbf X \mathbf C=\mathbf X \mathbf N \mathbf N^*\mathbf C^+.
\]

We construct the matrix $\mathbf N^*$ by taking a $53\times 53$ identity matrix indexed by conjugacy classes and removing the rows corresponding to one of each duplicate series and also the highest level conjugacy class (as ordered for instance by \Cref{mult}) appearing in each of the linear relations. The matrix $N$ may be constructed starting with a $35\times 35$ identity matrix, adding in columns to reconstitute the removed conjugacy classes and removing the columns corresponding to the theta series. 

The rows of $\mathbf N^*\mathbf C^+$ exhibit additional congruence as listed in \Cref{secCongruences}. For each prime $p$ dividing the order of the Thompson group, we construct a matrix $\mathbf M_p$ to reduce by these congruences. We start as before with a $35\times 35$ identity matrix. Then for each congruence listed we replace the row of the matrix with index  given by the highest weight conjugacy class appearing in the congruence with a new row constructed to reduce by that congruence. Given the congruence 
\[
\sum_{g} a_g\mathcal T_{[g]}\equiv 0\pmod{p^s}
\]
where the $a_g\in \Z,$ the new row will be given by 
\[ \sum_g p^{-s} a_g\mathfrak e_{[g]},
\]
where $\mathfrak e_{[g]}$ is the elementary basis element.

For instance, in the case $p=7,$ we have the congruence 
\[ \mathcal T_{[1A]}- \mathcal T_{[7A]} \equiv 0\pmod{7^2}.
\]
This tells us that the $[7A]$-th row of $M_7$ should be 
\[
7^{-2} \left(\mathfrak e_{[7A]}- \mathfrak e_{[1A]} \right).
\]

Assuming for the moment the validity of these congruences, we have that $\mathbf M_p\mathbf N^*\mathbf C^+$ is an integer matrix. Moreover, by construction $\mathbf M_p$ is invertible (depending on ordering, we have that $\mathbf M_p$ is lower-triangular with non-vanishing main diagonal). 

In each case we have computationally verified that the matrix
\[
\mathbf X \mathbf N \mathbf M_p^{-1}
\]
is rational with $p$-integral entries. Since 
\[
\mathbf m=\left(\mathbf X \mathbf N \mathbf M_p^{-1}\right)\cdot \left( \mathbf M_p\mathbf N^*\mathbf C^+\right).
\]
 is the product of two $p$-integral matrices, we have that every multiplicity must also be $p$-integral.
 
 The congruences listed in \Cref{secCongruences} were found computationally by reducing the matrix $\left(\mathbf N^*\mathbf C^+\right)\pmod p$ and computing the left kernel. After multiplying by a matrix constructed similar to $\mathbf M_p$ above so as to reduce by the congruences found, the process was repeated. The list of congruences given represents a complete list, in the sense that the matrix 
$\left(\mathbf M_p\mathbf N^*\mathbf C^+\right)$
both is integral and has full rank modulo $p$.
  
Many of the congruences can be easily proven using standard trace arguments for spaces of modular forms of level $pN$ to level $N$. For uniformity we will instead rely on Sturm's theorem following the argument described above. The worst case falls with any congruence involving the conjugacy class $24CD.$ These occur for both primes $p=2$ and $3$. The nature of the congruences, however, do not require us to increase the level beyond the corresponding level $N_{24CD}=1152$. There is a  unique normalized cusp form of weight $19/2$ and level $4$ in the plus space. This form vanishes to order $3$ at the cusp $\infty$ and to order $3/4$ at the cusp $1/2$. This is sufficient so that multiplying by this cusp form moves these potential congruences into spaces of holomorphic modular forms of weight $10,$ level $1152.$ The Sturm bound for this space falls just shy of $2\,000$ coefficients. This bound could certainly be reduced by more careful analysis, but this is sufficient for our needs. The congruences were observed up to $10\, 000$ coefficients. These computations were completed using \textsc{Sage} mathematical software\cite{sage}.

\begin{remark}
A similar process can be used in the case of Monstrous Moonshine to prove the integrality of the Monster character multiplicities. This gives a (probably\footnote{We say probably because the proof of Atkin--Fong--Smith relies on results in Margaret Ashworth's (later Millington) PhD thesis (Oxford University, 1964, advised by A. O. L. Atkin), of which the authors were unable to obtain a copy.}) alternate proof the theorem of Atkin--Fong--Smith \cite{Fong,Smith}. As in the case of Thompson moonshine, we have calculated a list of congruences for each prime dividing the order of the Monster, proven by means of Sturm's theorem. This is list complete in the sense that once we have reduced by the congruences for a given prime, the resulting forms have full rank modulo that prime. The Monster congruences may be of independent interest and are available upon request to the authors. 
\end{remark}
\section{Positivity of the multiplicities}\label{secPos}
To establish the positivity of the multiplicities of the irreducible representations we follow Gannon's work \emph{mutatis mutandis}. First we notice that by the first Schur orthogonality relation for characters and the triangle inequality, for each irreducible representation $\rho$ with corresponding character $\chi$ of $Th$ we have the estimate
\begin{align*}
\mult_{\rho_j}(W_k)&=\sum_{[g]\subseteq Th} \frac{1}{|C(g)|}\str_{W_k}(g)\overline{\chi(g)}\\
&\geq \frac{|\str_{W_k}(1)|}{|G|}\chi(1)-\sum_{[g]\neq [1]}\frac{|\str_{W_k}(g)|}{|C(g)|}|\chi(g)|.
\end{align*}
Here $C(g)$ denotes the centralizer of $g$ in $Th$ and the summation runs over all conjugacy classes of $Th$. Thus to prove positivity, we show that $|\str_{W_k}(1)|$ always dominates all the others. To this end, we use the description of the Fourier coefficients of $\calF_{[g]}(\tau)$ in terms of Maa{\ss}-Poincar\'e series, see \eqref{Z} and \eqref{calF}. We use the following elementary (and rather crude) estimates,
\begin{align*}
&\left|I_{\frac 12}(x)-\sqrt{\frac{2x}{\pi}}\right|\leq \frac{1}{5}\sqrt{\frac{2x^5}{\pi}}\quad \text{for }0<x<1,\\
&0<I_{\frac 12}(x)\leq \frac{e^x}{\sqrt{2\pi x}} \quad\text{for }x>0,\\
&|K_\psi(m,n,4c)|\leq c/2 \quad\text{for all }c\in\N,
\end{align*}
 and set $\delta_c=1+\dodd(Nc)$ which has the obvious bounds $1\leq \delta_c\leq 2$. 

The convergence and bounds of the coefficients $A_{N,\psi}$ rely on the  convergence of the modified \emph{Selberg-Kloosterman zeta function}
\begin{equation}\label{ModifiedZeta}
Z^*_\psi(m,n;s):=\sum_{c=1}^\infty (1-i)(1+\dodd(Nc))\frac{K_\psi(m,n,4Nc)}{(4Nc)^{2s}}.
\end{equation}\label{SKzeta}
The zeta function only converges conditionally at $3/4$. The bounds we obtain are crude and very large, but they do not grow with $n$. 
 
If we set 
\[C_{[g]}(n):=\frac{4N}{\pi \sqrt{2}}\left(\frac{n}{3}\right)^{1/4}A_{N,\psi}=\sum_{c=1}^{\infty}(1-i)(1+\dodd(Nc))\frac{K_{\psi}(-3,n,4Nc)}{4Nc}I_{\frac 12}\left(\frac{\pi\sqrt{3n}}{Nc}\right),
\]
then using the triangle inequality we find 
\[
\left|D\right|- \left|R\right|-\left|\sqrt8(3n)^{1/4}Z^*_\psi(m,n;3/4)\right|  
\leq |C(n)|
\leq \left|D\right|+\left|R\right| + \left|\sqrt8(3n)^{1/4}Z^*_\psi(m,n;3/4)\right|.
\]
Here $R$ is the absolutely convergent sum
\[\sum_{c=2}^{\infty}(1-i)(\delta_c)\frac{K_{\psi}(-3,n,4Nc)}{4Nc}\left(I_{\frac 12}\left(\frac{\pi\sqrt{3n}}{Nc}\right)-\sqrt{\frac{2\sqrt{3n}}{Nc}}\right),
\]
and $D$ is the dominant term coming from the $c=1$ term in the expression for $A_{N,\psi}.$
 We may estimate  $|D|$ by
\begin{align*}
|D|=&\left|(1-i)(\delta_1)\frac{K_{\psi}(-3,n,4N)}{4N}\left(I_{\frac 12}\left(\frac{\pi\sqrt{3n}}{N}\right)-\sqrt{\frac{2\sqrt{3n}}{N}}\right)\right|\\
\leq& \sqrt{2}\frac{\delta_1}{2}\frac{\sqrt{N}}{\pi\sqrt{2}(3n)^\frac 14}e^{\frac{\pi\sqrt{3n}}{N}}\\
\leq& \frac{2\sqrt{N}}{\pi(3n)^\frac 14}e^{\frac{\pi\sqrt{3n}}{N}}.
\end{align*}
Here we have used the second estimate for the Bessel function. We will only be interested in a \emph{lower} bound for $|D|$ when $[g]=[1A].$ In this case, we will just use the exact expression

\[
D_{[1A]}=(-1)^{n}\left(I_{\frac 12}\left(\pi\sqrt{3n}\right)-\sqrt{2}(3n)^{1/4}\right).
\]
 
If we set $L:=\tfrac{\pi}{N}\sqrt{3n}$, the sum for  $2\leq c\leq L$ in $R$ can be estimated as follows:
\begin{align*}
&\left|(1-i)\sum_{2\leq c\leq L}\delta_c\frac{K_{\psi}(-3,n,4Nc)}{4Nc}\left(I_{\frac 12}\left(\frac{\pi\sqrt{3n}}{Nc}\right)-\sqrt{\frac{2\sqrt{3n}}{Nc}}\right)\right|\\
\leq& \sqrt{2}\sum_{2\leq c\leq L}\frac{\delta_c}{2}\frac{\sqrt{Nc}}{\pi\sqrt{2}(3n)^\frac 14}e^{\frac{\pi\sqrt{3n}}{Nc}}
\leq \frac{\sqrt{N}}{\pi(3n)^\frac 14} L^\frac 32 e^{\frac{\pi\sqrt{3n}}{2N}}\\
=& \frac{\sqrt{3\pi n}}{N}e^{\frac{\pi\sqrt{3n}}{2N}}.
\end{align*}

For the terms of $R$ with $c\geq L$ we can use the first estimate on the Bessel function.
\begin{align*}
& \left|(1-i)\sum_{c> L}^{\infty}\delta_c\frac{K_{\psi}(-3,n,4Nc)}{4Nc}\left(I_{\frac 12}\left(\frac{\pi\sqrt{3n}}{Nc}\right)-\sqrt{\frac{2\sqrt{3n}}{Nc}}\right)\right|\\
\leq & \sqrt{2}\left|\sum_{c> L}^{\infty}\frac{\delta}{2}\frac{\sqrt{2}\pi^2 (3n)^\frac 54}{5(Nc)^\frac 54} \right|\\
\leq & \frac{2 \pi^2(3n)^\frac 54}{5N^\frac 52}\zeta\left(\frac 52\right)
\end{align*}

where $\zeta(s)$ denotes the Riemann zeta function. 
We now need only estimate $Z^*_\psi(m,n;\tfrac 34)$. 
 It turns out that we can use Gannon's estimates almost directly once we write the modified zeta function in a form sufficiently similar to the expressions he uses in his estimates. However we will need to slightly modify some of Harvey and Rayhaun's notation for $v$ and $h$. Let $\hat h = \frac{h}{(h,4)}$ and $\hat v=\frac{4\nu}{(h,4)}\pmod{\hat h}$ so that $\hat h\mid (N,24)$ and 
\[
\psi(4Nc,d)=\exp\left(-2\pi i \hat v \frac{cd}{\hat h}\right).
\] 
We note that in every case given we have that $\hat v\equiv \pm 1\pmod {\hat h}.$

With this notation we have that 
\begin{proposition}\label{GeneralThm3}
Let $n\geq 40$ with $D=mn$ a negative discriminant. 
The Selberg--Kloosterman zeta function defined in (\ref{SKzeta}) converges at $s=3/4,$ with the following bounds.

If $N\neq 2$ then
\begin{align*}
|Z^*_\psi(m,n;3/4)|\leq & \frac 14\left(\prod_{p\mid 4N\hat h}\left(1+\frac{1}{p}\right)\right)(1 + 2.13|D|^{1/8} \log |D|)\times\\
& \left((6.124N^{35/6}\hat h^{47/6} - 3.09N^{23/4}\hat h^{31/4} + 64.32N^{29/6}{\hat h}^7 - 23N^{19/4}\hat h^7)|D| \right.\\
&\left.~ +  ~(.146N^{47/6}\hat h^{65/6} - .114N^{31/4}\hat h^{43/4} + 2.51N^{35/6}\hat h^{10} - .74N^{23/4}\hat h^{10})|D|^{3/2}\right).
\end{align*}

If $N=2$, then 
\begin{align*}
|Z_\psi(m,n;3/4)|\leq &\frac 14\left(\prod_{p\mid 4N\hat h}\left(1+\frac{1}{p}\right)\right)(1 + 2.13|D|^{1/8} \log |D|)\times \left(3872|D| + 213|D|^{3/2}\right).
\end{align*}
\end{proposition}
\begin{proof}

Combining equations (\ref{Kloost}) and (\ref{psi}), we can write the Kloosterman sum $K_\psi(m,n,4Nc)$ as
\begin{equation}\label{Kloost2}
K_\psi(m,n,4Nc)=\sideset{}{^*}\sum_{d\smod{4Nc}}\left(\frac {4Nc}{d}\right)\eps_d\exp\left[2\pi i\left(\frac{m\overline{d}+(n-4\hat v c^2\cdot \frac{N}{\hat h})d}{4Nc}\right)\right].
\end{equation}
%
Using a result by Kohnen \cite[Proposition 5]{Kohnen}, we can write our Kloosterman sum as a sum over a more sparse set. Kohnen shows that
\begin{align}\label{KohnenFormula}
\frac{1}{\sqrt{4Nc}}(1-i)&\left(1+\dodd(Nc)\right)K(m,n,4Nc)\\
\nonumber&=\sum_{\substack{\beta \pmod{2Nc}\\ \beta^2\equiv mn\pmod{4Nc}}}\chi_{m}\left(\left[Nc,\beta,\frac{\beta^2-mn}{4Nc} \right]\right)\exp\left[2\pi i \left(\frac{\beta}{2Nc}\right)\right].
\end{align}
Here $[\alpha,\beta,\gamma]$ is a positive definite integral binary quadratic form, in this case with discriminant $mn$, and $\chi_\Delta$ is the genus character defined as follows on integral binary quadratic forms with discriminant divisible by $\Delta$ by
\[
\chi_\Delta\left(Q\right) =\chi_\Delta\left([\alpha,\beta,\gamma]\right):=\begin{cases}
\left(\frac{\Delta}{R}\right)& \text{ if} \operatorname{gcd}(\Delta, \alpha,\beta,\gamma)=1 \text{ where  } Q \text{ represents } R\\
0 &\text{otherwise.}
\end{cases}
\]


We can write equation (\ref{Kloost2}) in this form if we replace $n$ with 
$\tilde n=n-4\hat v c^2\cdot \frac{N}{\hat h}.$ Unfortunately, this makes the sum over a set of quadratic forms with discriminant $m\tilde n$ which depends on $c$. This is not ideal for approximating the zeta function. To fix this, notice that if the quadratic form $Q=[Nc,\beta,\gamma]$ has discriminant $m\tilde n,$ then the form $Q'=[Nc,\beta,\gamma'/\hat h]$ with $\gamma'=\frac{\beta^2-mn}{4Nc/\hat h}$ is a positive definite  binary quadratic form, with discriminant $mn$ and $\gamma'\equiv mv c\pmod{\hat h}.$  This relation defines a bijection between such forms.

Let $\mathcal Q_{N;\hat h,m \hat v }(D)$ denote the set of quadratic forms $Q=[Nc,\beta,\gamma/\hat h]$ of discriminant $D$ with $ c,\beta,\gamma\in \Z$ and $\gamma\equiv m\hat v c\pmod{\hat h},$ 
and let $\mathcal Q_{N}(d)$ denote the set of quadratic forms $Q=[Nc,\beta,\gamma]$ of discriminant $d$. Then we have the bijection
\[
\varphi_{\hat h, m \hat v}:\mathcal Q_{N;\hat h,m \hat v }(D)\to \mathcal Q_{N;}(D-4m \hat v \alpha^2N/\hat h)
\]
defined by 
\[
\varphi_{\hat h, m \hat v} [N\alpha,\beta,\gamma/\hat h] =[N\alpha,\beta,(\gamma-m\hat v \alpha)/\hat h].
\]
We may drop the subscript of $\varphi$ as it will generally be clear from context.

The set $\mathcal Q_{N;\hat h,m \hat v }(D)$ is acted upon by a certain matrix group which we denote by $\Gamma_0(N;h,m\hat v).$ This group consists of matrices  $\begin{pmatrix}a&b/h\\Nc&d\end{pmatrix}$ of determinant $1$ where each letter is an integer satisfying the relations 
$$a\equiv \ell d \pmod {\hat h} \ \ \ \ \text{ and } \ \ \ \ b\equiv \ell m\hat v c \pmod {\hat h}.$$
Here $\ell$ is some number coprime to $\hat h.$ This generalizes the groups  $\Gamma_0(N;h)=\Gamma_0(N;h,1)$ used by Gannon.

\begin{proposition}\label{General5b}
Assume the notation above, and let $Q_1,Q_2\in \mathcal Q_{N;\hat h,\Delta \hat v }(D)$ where $\Delta,D$ are discriminants, with $\Delta$ fundamental and $D$ divisible by $\Delta$. If $Q_1$ and $Q_2$ are related by the action of some $M\in \Gamma_0(N;h,m\hat v),$ then 
\[
\chi_\Delta\left(\varphi 
 Q_1\right)=\chi_\Delta\left(\varphi 
  Q_2\right)
\] 
\end{proposition}
\begin{proof}
Since $\Delta$ is a discriminant, the definition of $\chi_\Delta(Q)$ as a Kronecker symbol allows us to reduce the coefficients of the quadratic form $Q$ modulo $\Delta.$ It is also multiplicative. If $\Delta=\Delta_1\Delta_2$ is a factorization into discriminants, then 
\[
\chi_\Delta(Q)=\chi_{\Delta_1}(Q)\chi_{\Delta_2}(Q).
\]

We will want a factorization of $\Delta$ into discriminants $\Delta=\Delta'\Delta_h$ where $(\Delta',\hat h)=1$ and $\Delta_h$ is divisible only by primes dividing $\hat h.$ Since $\Delta$ is a fundamental discriminant and $\hat h$ divides 24, this means that $|\Delta_h|$ also divides $24.$   

By construction, $\varphi ~Q_1\equiv Q_1\pmod{\Delta'}.$ As a determinant $1$ matrix will not alter the integers represented by a quadratic form, we have that 

\[\chi_{\Delta'}(\varphi ~Q_1)=\chi_{\Delta'}(Q_1)=\chi_{\Delta'}(Q_2)=\chi_{\Delta'}(\varphi ~Q_2)
\]

Suppose $Q_1=[N\alpha,\beta, \gamma/\hat h]$ and  $Q_2=Q_1|M$ where $M$ is the matrix $\begin{pmatrix}a & b/ \hat h\\Nc&d\end{pmatrix}$ with $a\equiv \ell d \pmod{\hat h}.$ 
If  we set $\gamma'= (\gamma-m\hat v \alpha)/\hat h,$ then a short calculation show that 
\[\varphi Q_2\equiv [N(\alpha a^2 +\beta ac), \beta ad, d^2\gamma'] \pmod{\Delta}.
\]
Since $M$ has determinant $1$, we have that $a$ and $d$ are coprime to $\Delta.$ Since $\Delta_h\mid 24,$ we have that $a^2\equiv d^2\equiv 1\pmod{\Delta_h}.$ Moreover, by considering the discriminant  we see that $\beta$ is even if $\Delta_h$ is, and  $4\mid \beta$ if $8\mid \Delta_h.$ In either case, we find that $\Delta_h $ divides $\beta N$. Therefore we may further reduce to 
\[
\varphi Q_2\equiv [N\alpha, \beta \ell, \gamma'] \pmod{\Delta_h}.
\]
 The $\ell$ does not change the possible numbers represented, and so we have that 
\[\chi_{\Delta_h}(\varphi ~Q_1)=\chi_{\Delta_h}(\varphi ~Q_2),
\]
concluding the proof of proposition \ref{General5b}.
\end{proof}

Once again following Gannon
, if we combine equations (\ref{ModifiedZeta}), (\ref{KohnenFormula}) and proposition \ref{General5b}, we can write the modified Selberg--Kloosterman zeta function at $s=3/4$ as 
\begin{align*}
Z^*_\psi(m,n;3/4)=\lim_{X\to \infty}&\sum_{[Q]\in  \mathcal Q_{N;\hat h,\hat v}(mn) / \Gamma_0(N;\hat h,m\hat v)}\frac{\chi_m(\varphi_{\hat h, m \hat v}  Q)}{\omega_Q}\\
&\sum_{\substack{r,Ns\in Z,\\ (r,Ns/h)=1\\ 0<c(Q,r,Ns)<X}}
  \frac{1}{4Nc(Q,r,Ns)} \exp\left(2\pi i \frac{\beta}{2Nc(Q,r,Ns)}\right).
\end{align*}
Here $Q=[N\alpha,\beta,\gamma/\hat h],$  
 $\omega_Q$ is the order of the stabilizer of $Q$ in $\Gamma_0(N;\hat h,m\hat v),$ and $c(Q,r,Ns)=\frac{Q(r,Ns)}{4N}.$

  This equation is analogous to Equation (4.26) of \cite{Gannon}, but differs in four main points: First, we have normalized the zeta function slightly differently. Second, the bijection $\varphi_{\hat h, m\hat v}$ and proposition \ref{General5b} give a more general version of Gannon's Lemma 5(b) allowing us to sum over $\mathcal Q_{N;\hat h,\hat v}(mn) / \Gamma_0(N;\hat h,m\hat v)$ rather than $\mathcal Q_{N;\hat h}(mn) / \Gamma_0(N;\hat h)$. Third, Gannon's case was restricted to discriminants where the stabilizer could only be $\{\pm I\},$ and so he replaces the $\omega_Q$ term with a $2$ in his equation. We will use this as a lower bound for $\omega_Q.$ 
 Fourth, his sum contains a power of $-1$ while ours contains a genus character. In either case, the sign is constant for a given representative quadratic form $Q$.
 
  Gannon estimates the inner sums in absolute value and the outer sum by bounding the number of classes of quadratic forms. His bounds for the size of $\mathcal Q_{N;\hat h}(mn) / \Gamma_0(N;\hat h)$ are crude enough to also hold for the number of classes of $\mathcal Q_{N;\hat h,\hat v}(mn) / \Gamma_0(N;\hat h,m\hat v).$ 
  Proposition \ref{GeneralThm3} follows from using Gannon's bounds modified only to account for our differences in normalization.
\end{proof}

Combining these estimates as described above, we find that each multiplicity of the irreducible components of $W_n$ must alway be positive for $n\geq 375.$ Explicit calculations up to $n=375$ show that these multiplicities are always positive. The worst cases for the estimates with  $n\leq 375$ arise from the trivial character or from estimating for Selberg--Kloosterman zeta function for the $24CD$ conjugacy class. These calculations were performed using \textsc{Sage} mathematical software \cite{sage}. 

\section{Replicability}\label{secRepl}
One important property of the Hauptmoduln occurring in Monstrous Moonshine is that they are \emph{replicable}. 
\begin{definition}\label{repldef}
Let $f(\tau)=q^{-1}+\sum_{n=0}^\infty H_nq^n$ be a (formal) power series with integer coefficients and consider the function
\[F(\tau_1,\tau_2)=\log(f(\tau_1)-f(\tau_2))=\log(q_1^{-1}-q_2^{-1})-\sum_{m,n=1}^\infty H_{m,n}q_1^mq_2^n,\]
where $\tau_1,\tau_2\in\HH$ are two independent variables and $q_j=e^{2\pi i\tau_j}$, $j=1,2$. We call $f$ \emph{replicable}, if we have that $H_{a,b}=H_{c,d}$ whenever $ab=cd$ and $\gcd(a,b)=\gcd(c,d)$.
\end{definition}
This property of the Hauptmoduln involved in Monstrous Moonshine in a sense reflects the algebra structure of the Monstrous-Moonshine module, see \cite{CN79,Nort84}.

An important, but not immediately obvious fact is that any replicable function is determined by its first 23 Fourier coefficients, \cite{FMN,Nort84}.
\begin{theorem}
Let $f(\tau)=q^{-1}+\sum_{n=1}^\infty a_nq^n$ be a replicable function. Then one can compute the coefficient $a_n$ for any $n\in\N$ constructively out of the coefficients 
\[\{a_1,a_2,a_3,a_4,a_5,a_7,a_8,a_9,a_{11},a_{17},a_{19},a_{23}\}.\]
\end{theorem}
A \textsc{Maple} procedure to perform this computation is printed at the end of \cite{FMN}.

In \cite{DGO}, there is also an analogous notion of replicability in the mock modular sense, which requires that the Fourier coefficients satisfy a certain type of recurrence. This is a special phenomenon occuring for mock theta functions, i.e., mock modular forms whose shadow is a unary theta function, satisfying certain growth conditions at cusps, see \cite{IRR13, Mert}.

It is now natural to ask about replicability properties of the McKay-Thompson series in the case of Thompson Moonshine. Let $g$ be any element of the Thompson group and $\calT_{[g]}(\tau)=\calF_{[g]}(\tau)$ be the corresponding McKay-Thompson series as in \eqref{calF}. As we have seen, this is a weekly holomorphic modular form of weight $\tfrac 12$ living in the Kohnen plus space. To relate these to the Hauptmoduln and other replicable functions discussed in \cite{FMN}, we split into an even and an odd part
\[\calT_{[g]}^{(0)}(\tau)=\sum_{m=0}^\infty \str_{W_{4m}}q^{4m}\quad\text{and}\quad \calT_{[g]}^{(0)}(\tau)=\sum_{m=0}^\infty \str_{W_{4m-3}}q^{4m-3}\]
in the notation of \Cref{conj:main}. Letting 
\[\vartheta^{(0)}(\tau)=\vartheta(\tau)=\frac{\eta(2\tau)^5}{\eta(\tau)^2\eta(4\tau)^2}=\sum_{n\in\Z}q^{n^2}\]
and 
\[\vartheta^{(1)}(\tau)=2\frac{\eta(4\tau)^2}{\eta(2\tau)}=\sum_{n\in\Z}q^{\left(n+\frac 12\right)^2}\]
we define the weight $0$ modular functions
\[\t^{(j)}_{[g]}(\tau)=\frac{\calT_{[g]}^{(j)}\left(\frac\tau4\right)}{\vartheta^{(j)}(\tau)},\quad (j=0,1) \]
and we set $\t_{[g]}(\tau)=\t^{(0)}_{[g]}(\tau)+\t^{(1)}_{[g]}(\tau)$. Note that these functions don't have poles in $\HH$.

As it turns out through direct inspection, these weight $0$ functions are often replicable functions or univariate rational functions therein. We used the list of replicable functions given in \cite{FMN} as a reference and found the identities given in Tables \ref{repl} -- \ref{repliii}, which are all identities of the given form using the aforementioned table of replicable functions at the end of \cite{FMN} and allowing the degree of the denominator of the rational function to be as large as $40$.
\newpage
\appendix
\section{Tables}
Here we give the character table of the Thompson group $Th$ over the complex numbers. We set
\[\begin{array}{lll}
A:=-1+2i\sqrt{3},\qquad&
B:=-2+4i\sqrt{3},\qquad &
C:=\frac{-1+i\sqrt{15}}{2}, \\
D:=-i\sqrt{3},&
E:=-i\sqrt{6},&
F:=\frac{-1+3i\sqrt{3}}{2},\\
G:=\frac{-1+i\sqrt{31}}{2},&
H:=-1+i\sqrt{3},&
I:=\frac{-1+i\sqrt{39}}{2},
\end{array}\]
and overlining one of these characters denotes complex conjugation. We used \textsc{Gap4} \cite{gap} to find the character table.
\begin{landscape}
\begin{center}
\begin{table}[htp]
\small
\setlength{\tabcolsep}{2.5pt}
\begin{tabular}{|r|rrrrrrrrrrrrrrrrrrrrrrrr|}
\hline
     &   1A &    2A  &    3A   &  3B  &   3C   &    4A    &  4B  &  5A &   6A  &  6B  &  6C  &   7A  &  8A  &  8B  &   9A & 9B & 9C & 10A & 12A & 12B & 12C & 12D & 13A & 14A\\
     \hline
$\chi_1$    &      1  &   1  &   1 &   1 &   1  &  1 &  1 &  1 &  1 &  1 &  1 &  1 & 1 & 1 &  1 &
		           1  &   1  &   1 &   1 &   1  &  1 &  1 &  1 &  1 \\
$\chi_2$    &    248  &  -8  &  14 &   5 &  -4  &  8 &  0 & -2 &  4 & -2 &  1 &  3 & 0 & 0 &  5 &
				  -4  &   2  &   2 &   2 &   2  & -1 &  0 &  1 & -1 \\
$\chi_3$    &   4123  &  27  &  64 &  -8 &   1  & 27 & -5 & -2 &  9 &  0 &  0 &  7 & 3 &-1 & -8 & 
				   1  &   4  &   2 &   0 &   0  &  0 &  1 &  2 & -1 \\
$\chi_4$    &  27000  & 120  & -27 &  27 &   0  &  8 &  0 &  0 &  0 & -3 &  3 &  1 & 0 & 0 &  0 &
			       0  &   0  &   0 &  $A$ &  $\overline{A}$ &  -1 &  0 &  -1 &  1\\
$\chi_5$    &  27000  & 120  & -27 &  27 &   0  &  8 &  0 &  0 &  0 & -3 &  3 &  1 & 0 & 0 &  0 &
				   0  &   0  &   0 &  $\overline{A}$ &  $A$ &  -1 &  0 &  -1 &  1\\
$\chi_6$    &  30628  & -92  &  91 &  10 &  10  & 36 &  4 &  3 & 10 & -5 & -2 &  3 &-4 & 0 & 10 &
			      10  &   1  &   3 &   3 &   3  &  0 & -2 &  0 & -1 \\
$\chi_7$    &  30875  & 155  & 104 &  14 &   5  & 27 & -5 &  0 &  5 &  8 &  2 &  5 & 3 &-1 & 14 &
			       5  &   2  &   0 &   0 &   0  &  0 &  1 &  0 &  1 \\
$\chi_8$    &  61256  &  72  & 182 &  20 &  20  & 56 &  0 &  6 & 12 &  6 &  0 &  6 & 0 & 0 & -7 &
				  -7  &   2  &   2 &   2 &   2  &  2 &  0 &  0 &  2 \\
$\chi_9$    &  85995  & -21  &   0 & -27 &  27  &-21 & 11 & -5 &  3 &  0 & -3 &  0 & 3 &-1 &  0 & 
                   0  &   0  &  -1 &   0 &   0  & -3 & -1 &  0 &  0 \\
$\chi_{10}$ &  85995  & -21  &   0 & -27 &  27  &-21 & 11 & -5 &  3 &  0 & -3 &  0 & 3 &-1 &  0 & 
                   0  &   0  &  -1 &   0 &   0  & -3 & -1 &  0 &  0 \\
$\chi_{11}$ & 147250  &  50  & 181 &  -8 & -35  & 34 & 10 &  0 &  5 &  5 & -4 &  5 & 2 &-2 & 19 & 
				  -8  &   1  &   0 &   1 &   1  & -2 &  1 & -1 &  1 \\
$\chi_{12}$ & 767637  & 405  &   0 &   0 &   0  &-27 & -3 & 12 &  0 &  0 &  0 &  3 &-3 &-3 &  0 & 
				   0  &   0  &   0 &   0 &   0  &  0 &  0 &  0 & -1 \\
$\chi_{13}$ & 767637  & 405  &   0 &   0 &   0  &-27 & -3 & 12 &  0 &  0 &  0 &  3 &-3 &-3 &  0 &
 				   0  &   0  &   0 &   0 &   0  &  0 &  0 &  0 & -1 \\
$\chi_{14}$ & 779247  &-273  &-189 & -54 &   0  & 63 & -9 & -3 &  0 &  3 &  6 &  0 &-1 & 3 &  0 & 
				   0  &   0  &  -3 &   3 &   3  &  0 &  0 &  1 &  0 \\
$\chi_{15}$ & 779247  &-273  &-189 & -54 &   0  & 63 & -9 & -3 &  0 &  3 &  6 &  0 &-1 & 3 &  0 &
				   0  &   0  &  -3 &   3 &   3  &  0 &  0 &  1 &  0 \\
$\chi_{16}$ & 957125  &-315  & 650 & -52 & -25  &133 &  5 &  0 & 15 & -6 &  0 &  8 &-3 & 1 &-25 &
				   2  &   2  &   0 &  -2 &  -2  & -2 & -1 &  0 &  0 \\
$\chi_{17}$ &1707264  &-768  &   0 & -54 &  54  &  0 &  0 & 14 & -6 &  0 &  6 &  6 & 0 & 0 &  0 &
			       0  &   0  &   2 &   0 &   0  &  0 &  0 &  0 &  2 \\
$\chi_{18}$ &1707264  &-768  &   0 & -54 &  54  &  0 &  0 & 14 & -6 &  0 &  6 &  6 & 0 & 0 &  0 &
      			   0  &   0  &   2 &   0 &   0  &  0 &  0 &  0 &  2 \\
$\chi_{19}$ &2450240  & 832  & 260 &  71 &  44  & 64 &  0 &-10 &  4 &  4 & -5 & -5 & 0 & 0 & 17 &
				 -10  &  -1  &   2 &   4 &   4  &  1 &  0 &  0 & -1 \\
$\chi_{20}$ &2572752  &-1072 & 624 & 111 &  84  & 48 &  0 &  2 & -4 &-16 & -1 &  7 & 0 & 0 & 30 &
				   3  &    3 &  -2 &   0 &   0  &  3 &  0 &  0 & -1 \\
$\chi_{21}$ &3376737  & 609  & 819 &   9 &   9  &161 &  1 &-13 &  9 &  3 & -3 &  0 & 1 & 1 &  9 &
				   9  &   0  &  -1 &  -1 &  -1  & -1 &  1 &  0 &  0 \\
$\chi_{22}$ &4096000  &   0  &  64 &  -8 & -80  &  0 &  0 &  0 &  0 &  0 &  0 & -8 & 0 & 0 & -8 &
				   1  &   4  &   0 &   0 &   0  &  0 &  0 & -1 &  0 \\
$\chi_{23}$ &4096000  &   0  &  64 &  -8 & -80  &  0 &  0 &  0 &  0 &  0 &  0 & -8 & 0 & 0 & -8 &
				   1  &   4  &   0 &   0 &   0  &  0 &  0 & -1 &  0 \\
$\chi_{24}$ &4123000  & 120  & 118 &  19 & -80  &  8 &  0 &  0 &  0 &  6 &  3 & -7 & 0 & 0 & 19 & 
				   1  &   4  &   0 &   2 &   2  & -1 &  0 & -2 &  1 \\
\hline
\end{tabular}
\caption{Character table of $Th$, Part I.}
\label{char1}
\end{table}
\end{center}
\end{landscape}
\newpage
\begin{landscape}
\begin{center}
\begin{table}[htp]
\small
\setlength{\tabcolsep}{2.5pt}
\begin{tabular}{|r|rrrrrrrrrrrrrrrrrrrrrrrr|}
\hline
     &   1A &    2A  &    3A   &  3B  &   3C   &    4A    &  4B  &  5A &   6A  &  6B  &  6C  &   7A  &  8A  &  8B  &   9A & 9B & 9C & 10A & 12A & 12B & 12C & 12D & 13A & 14A\\
     \hline
$\chi_{25}$ &   4881384 &  1512 &  729 &   0 &   0 &  72 & 24 &  9 &  0 &  9 & 0 &  4 & 8 & 0 &  0 &
 					  0 &     0 &   -3 &  -3 &  -3 &   0 &  0 &  1 &  0 \\			
$\chi_{26}$ &   4936750 &  -210 &  637 & -38 & -65 & 126 &-10 &  0 & 15 & -3 & 6 &  0 &-2 & 2 & 16 &
					-11 &    -2 &    0 &  -3 &  -3 &   0 & -1 &  0 &  0 \\
$\chi_{27}$ &   6669000 & -1080 & -351 & 108 &   0 &  56 &  0 &  0 &  0 &  9 & 0 &  2 & 0 & 0 &  0 &
					  0 &     0 &    0 &  $A$ & $\overline{A}$ &   2 &  0   & 0 &  -2 \\
$\chi_{28}$ &   6669000 & -1080 & -351 & 108 &   0 &  56 &  0 &  0 &  0 &  9 & 0 &  2 & 0 & 0 &  0 &
			          0 &     0 &    0 &  $\overline{A}$ & $A$ &   2 &  0   & 0 &  -2 \\
$\chi_{29}$ &   6696000 &  -960 & -378 & 135 &   0 &  64 &  0 &  0 &  0 &  6 & 3 &  3 & 0 & 0 &  0 &
					  0 &     0 &    0 &  $B$ & $\overline{B}$  &  1 & 0 & -1 & -1 \\ 
$\chi_{30}$ &   6696000 &  -960 & -378 & 135 &   0 &  64 &  0 &  0 &  0 &  6 & 3 &  3 & 0 & 0 &  0 &
					  0 &     0 &    0 &  $\overline{B}$ & $B$  &  1 & 0 & -1 & -1 \\ 
$\chi_{31}$ &  10822875 &  -805 &  924 & 141 & -75 &  91 & -5 &  0 &  5 & -4 & 5 &  0 & 3 &-1 &-21 &
					  6 &    -3 &    0 &   4 &   4 &   1 &  1 & -2 &  0 \\ 
$\chi_{32}$ &  11577384 &   552 &  351 & 135 &   0 &-120 & 24 &  9 &  0 & 15 & 3 &  7 &-8 & 0 &  0 &
					  0 &     0 &   -3 &   3 &   3 &  -3 &  0 &  0 & -1 \\
$\chi_{33}$ &  16539120 &  2544 &    0 & 297 & -54 &  48 & 16 & -5 & -6 &  0 &-3 &  3 & 0 & 0 &  0 &
					  0 &     0 &   -1 &   0 &   0 &   3 & -2 &  0 &  3 \\
$\chi_{34}$ &  18154500 &  1540 & -273 & 213 & -30 & -28 & 20 &  0 & 10 &-17 & 1 &  0 &-4 & 0 & -3 &
					 -3 &    -3 &    0 &  -1 &  -1 &  -1 &  2 &  0 &  0 \\ 
$\chi_{35}$ &  21326760 &   168 &    0 &-135 &-108 &-168 &  0 & 10 & 12 &  0 &-3 &  0 & 0 & 0 &  0 &
					  0 &     0 &   -2 &   0 &   0 &   3 &  0 &  0 &  0 \\ 
$\chi_{36}$ &  21326760 &   168 &    0 &-135 &-108 &-168 &  0 & 10 & 12 &  0 &-3 &  0 & 0 & 0 &  0 &
					  0 &     0 &   -2 &   0 &   0 &   3 &  0 &  0 &  0 \\ 
$\chi_{37}$ &  28861000 &   840 & 1078 &-110 & 160 &  56 &  0 &  0 &  0 &  6 &-6 &  0 & 0 & 0 &-29 &
					 -2 &    -2 &    0 &   2 &   2 &   2 &  0 & -1 &  0 \\
$\chi_{38}$ &  30507008 &     0 &  896 &-184 &  32 &   0 &  0 &  8 &  0 &  0 & 0 &  0 & 0 & 0 & 32 &
					  5 &    -4 &    0 &   0 &   0 &   0 &  0 & -1 &  0 \\ 
$\chi_{39}$ &  40199250 &  3410 &  -78 &   3 & 165 & -62 & 10 &  0 &  5 &  2 &-1 & -7 &-6 & 2 &  3 &
					  3 &     3 &    0 &  -2 &  -2 &   1 &  1 &  0 &  1 \\
$\chi_{40}$ &  44330496 &  3584 &  168 &   6 &-156 &   0 &  0 & -4 & -4 &  8 & 2 &  0 & 0 & 0 &  6 &
					  6 &    -3 &    4 &   0 &   0 &   0 &  0 &  2 &  0 \\
$\chi_{41}$ &  51684750 &  2190 &    0 & 108 & 135 &-162 &-10 &  0 & 15 &  0 &12 & -9 & 6 &-2 &  0 &
					  0 &     0 &    0 &   0 &   0 &   0 & -1 &  0 & -1 \\
$\chi_{42}$ &  72925515 & -2997 &    0 &   0 &   0 &  27 & 51 & 15 &  0 &  0 & 0 & -9 & 3 & 3 &  0 &
					  0 &     0 &    3 &   0 &   0 &   0 &  0 &  0 & -1 \\
$\chi_{43}$ &  76271625 & -2295 &  729 &   0 &   0 & 153 &-15 &  0 &  0 &  9 & 0 &-11 &-7 &-3 &  0 &
					  0 &     0 &    0 &  -3 &  -3 &   0 &  0 &  1 &  1 \\
$\chi_{44}$ &  77376000 &  2560 & 1560 & -60 & -60 &   0 &  0 &  0 &-20 & -8 & 4 &  2 & 0 & 0 & -6 &
					  0 &     0 &   -3 &   3 &   3 &   0 &  0 &  2 &  1 \\ 
$\chi_{45}$ &  81153009 &  -783 & -729 &   0 &   0 & 225 &  9 &  9 &  0 & -9 & 0 & -7 & 1 &-3 &  0 &
					  0 &     0 &   -3 &   3 &   3 &   0 &  0 &  2 &  1 \\
$\chi_{46}$ &  91171899 &   315 &    0 & 243 &   0 & -21 &-45 & 24 &  0 &  0 &-9 &  0 & 3 & 3 &  0 &
					  0 &     0 &    0 &   0 &   0 &  -3 &  0 &  0 &  0 \\ 
$\chi_{47}$ & 111321000 &  3240 &-1728 &-216 &   0 & 216 &  0 &  0 &  0 &  0 & 0 &  7 & 0 & 0 &  0 &
					  0 &     0 &    0 &   0 &   0 &   0 &  0 & -2 & -1 \\
$\chi_{48}$ & 190373976 & -3240 &    0 &   0 &   0 &-216 &  0 &-24 &  0 &  0 & 0 &  9 & 0 & 0 &  0 &
					  0 &     0 &    0 &   0 &   0 &   0 &  0 &  0 &  1 \\
\hline
\end{tabular}
\caption{Character table of $Th$, Part II.}
\label{char2}
\end{table}
\end{center}
\end{landscape}

\newpage
\begin{landscape}
\begin{center}
\begin{table}[htp]
\small
\setlength{\tabcolsep}{2.5pt}
\begin{tabular}{|r|rrrrrrrrrrrrrrrrrrrrrrrr|}
\hline
 & 15A & 15B & 18A & 18B & 19A & 20A & 21A & 24A & 24B & 24C & 24D & 27A & 27B & 27C & 28A & 30A & 30B & 31A & 31B & 36A & 36B & 36C & 39A & 39B \\
 \hline
$\chi_{1}$  &       1 &  1 &  1 &  1 &  1 &  1 &  1 &  1 &  1 &  1 &  1 &  1 &  1 &  1 &  1 &  1 &   1 &  1 &  1 &  1 &  1 &  1 &  1 &  1 \\
$\chi_{2}$  &       1 &  1 &  1 & -2 &  1 &  0 &  0 &  0 &  0 &  0 &  0 &  2 & -1 & -1 &  1 & -1 & -1 &  0 &  0 & -1 &  -1 & -1 &  1 &  1 \\
$\chi_{3}$  &     1 &  1 &  0 &  0 &  0 &  0 &  1 &  0 &  0 & -1 & -1 & -2 &  1 &  1 & -1 & -1 & -1 &  0 &  0 &  0 &  0 &  0 & -1 &  -1 \\
$\chi_{4}$  &       0 &  0 &  0 &  0 &  1 &  0 &  1 &  $D$ & $-D$ &  0 &  0 &  0 &  0 &  0 &  1 &  0 &  0 & -1 & -1 &  2 & $H$ & $\overline{H}$ &  -1 &  -1 \\
$\chi_{5}$  &       0 &  0 &  0 &  0 &  1 &  0 &  1 & $-D$ &  $D$ &  0 &  0 &  0 &  0 &  0 &  1 &  0 &  0 & -1 & -1  & 2 & $\overline{H}$ & $H$ &  -1 & -1 \\
$\chi_{6}$  &       0 &  0 & -2 &  1 &  0 & -1 &  0 & -1 & -1 &  0 &  0 &  1 &  1 &  1 &  1 &  0 &  0 &  0 &  0 &  0 &  0 &  0 &  0 &  0 \\
$\chi_{7}$  &       0 &  0 &  2 &  2 &  0 &  0 & -1 &  0 &  0 & -1 & -1 &  2 & -1 & -1 & -1 &  0 &  0 & -1 & -1 &  0 &  0 &  0 &  0 &  0 \\
$\chi_{8}$  &        0 &  0 & -3 &  0 &  0 &  0 &  0 &  0 &  0 &  0 &  0 & -1 & -1 & -1 &  0 &  2 &  2 &  0 &  0 & -1 & -1 & -1 &  0 &  0 \\
$\chi_{9}$  &       $C$ & $\overline{C}$ &   0 &  0 &  1 &  1 &  0 &  0 &  0 & -1 & -1 &  0 &  0 &  0 &  0 & -$C$ & $-\overline{C}$ &   1 &  1  & 0 &  0 &  0 &  0 &  0 \\
$\chi_{10}$ &     $\overline{C}$ &   $C$ &   0 &  0 &  1 &  1 &  0 &  0 &  0 & -1 & -1 &  0 &  0 &  0 &  0 & $-\overline{C}$ &  $-C$ &  1 &  1 &  0 &  0 &  0 &  0 &  0 \\
$\chi_{11}$ &      0 &  0 & -1 & -1 &  0 &  0 & -1 & -1 & -1 &  1 &  1 &  1 &  1 &  1 & -1 &  0 &  0 &  0 &  0 &  1 &  1 &  1 & -1 & -1 \\
$\chi_{12}$ &      0 &  0 &  0 &  0 & -1 &  2 &  0 &  0 &  0 &  0 &  0 &  0 &  0 &  0 &  1 &  0 &  0 &  $G$ &  $\overline{G}$ &  0 &  0 &  0 &  0 &  0 \\
$\chi_{13}$ &      0 &  0 &  0 &  0 & -1 &  2 &  0 &  0 &  0 &  0 &  0 &  0 &  0 &  0 &  1 &  0 &  0 & $\overline{G}$ &  $G$  &  0 &  0 &  0 &  0 &  0 \\
$\chi_{14}$ &      0 &  0 &  0 &  0 &  0 &  1 &  0 & -1 & -1 &  0 &  0 &  0 &  0 &  0 &  0 &  0 &  0 &  0 &  0 &  0 &  0 &  0 &  $I$ & $\overline{I}$ \\
$\chi_{15}$ &      0 &  0 &  0 &  0 &  0 &  1 &  0 & -1 & -1 &  0 &  0 &  0 &  0 &  0 &  0 &  0 &  0 &  0 &  0 &  0 &  0 &  0 & $\overline{I}$ & $I$ \\
$\chi_{16}$ &      0 &  0 &  3 &  0 &  0 &  0 & -1 &  0 &  0 &  1 &  1 & -1 & -1 & -1 &  0 &  0 &  0 &  0 &  0 &  1 &  1 &  1 &  0 &  0 \\
$\chi_{17}$ &     -1 & -1 &  0 &  0 &  0 &  0 &  0 &  0 &  0 &  $E$ & $-E$ &  0 &  0 &  0 &  0 & -1 & -1 &  1 &  1 &  0 &  0 &  0 &  0 &  0 \\
$\chi_{18}$ &     -1 & -1 &  0 &  0 &  0 &  0 &  0 &  0 &  0 & $-E$ &  $E$ &  0 &  0 &  0 &  0 & -1 & -1 &  1 &  1 &  0 &  0 &  0 &  0 &  0 \\
$\chi_{19}$ &     -1 & -1 &  1 &  1 &  0 &  0 &  1 &  0 &  0 &  0 &  0 & -1 & -1 & -1 &  1 & -1 & -1 &  0 &  0 &  1 &  1 &  1 &  0 &  0 \\
$\chi_{20}$ &     -1 & -1 &  2 & -1 &  0 &  0 &  1 &  0 &  0 &  0 &  0 &  0 &  0 &  0 & -1 &  1 &  1 &  0 &  0 &  0 &  0 &  0 &  0 &  0 \\
$\chi_{21}$ &     -1 & -1 & -3 &  0 &  0 &  1 &  0 &  1 &  1 &  1 &  1 &  0 &  0 &  0 &  0 & -1 & -1 &  0 &  0 & -1 & -1 & -1 &  0 &  0 \\
$\chi_{22}$ &      0 &  0 &  0 &  0 & -1 &  0 &  1 &  0 &  0 &  0 &  0 &  1 &  $F$ & $\overline{F}$ &  0 &  0 &  0 &  1 &  1 &  0 &  0 &  0 & -1 & -1 \\
$\chi_{23}$ &      0 &  0 &  0 &  0 & -1 &  0 &  1 &  0 &  0 &  0 &  0 &  1 & $\overline{F}$ &  $F$ &  0 &  0 &  0 &  1 &  1 &  0 &  0 &  0 & -1 & -1 \\
$\chi_{24}$ &      0 &  0 &  3 &  0 &  0 &  0 & -1 &  0 &  0 &  0 &  0 & -2 &  1 &  1 &  1 &  0 &  0 &  0 &  0 & -1 & -1 & -1 &  1 &  1 \\
\hline
\end{tabular}
\caption{Character table of $Th$, Part III.}
\label{char3}
\end{table}
\end{center}
\end{landscape}

\newpage
\begin{landscape}
\begin{center}
\begin{table}[htp]
\small
\setlength{\tabcolsep}{2.5pt}
\begin{tabular}{|r|rrrrrrrrrrrrrrrrrrrrrrrr|}
\hline
 & 15A & 15B & 18A & 18B & 19A & 20A & 21A & 24A & 24B & 24C & 24D & 27A & 27B & 27C & 28A & 30A & 30B & 31A & 31B & 36A & 36B & 36C & 39A & 39B \\
 \hline
$\chi_{24}$ &      0 &  0 &  3 &  0 &  0 &  0 & -1 &  0 &  0 &  0 &  0 & -2 &  1 &  1 &  1 &  0 &  0 &  0 &  0 & -1 & -1 & -1 &  1 &  1 \\
$\chi_{25}$ &      0 &  0 &  0 &  0 & -1 & -1 &  1 & -1 & -1 &  0 &  0 &  0 &  0 &  0 &  2 &  0 &  0 &  0 &  0 &  0 &  0 &  0 &  1 &  1 \\
$\chi_{26}$ &      0 &  0 &  0 &  0 & -1 &  0 &  0 &  1 &  1 & -1 & -1 &  1 &  1 &  1 &  0 &  0 &  0 &  0 &  0 &  0 &  0 &  0 &  0 &  0 \\
$\chi_{27}$ &      0 &  0 &  0 &  0 &  0 &  0 & -1 & $-D$ &  $D$ &  0 &  0 &  0 &  0 &  0 &  0 &  0 &  0 &  1 &  1 &  2 &  $H$ & $\overline{H}$ &  0 &  0 \\
$\chi_{28}$ &      0 &  0 &  0 &  0 &  0 &  0 & -1 &  $D$ & $-D$ &  0 &  0 &  0 &  0 &  0 &  0 &  0 &  0 &  1 &  1 &  2 & $\overline{H}$ &   $H$ &  0 &  0 \\
$\chi_{29}$ &      0 &  0 &  0 &  0 &  1 &  0 &  0 &  0 &  0 &  0 &  0 &  0 &  0 &  0 &  1 &  0 &  0 &  0 &  0 & -2 & $-H$ & $-\overline{H}$ &  -1 & -1 \\
$\chi_{30}$ &      0 &  0 &  0 &  0 &  1 &  0 &  0 &  0 &  0 &  0 &  0 &  0 &  0 &  0 &  1 &  0 &  0 &  0 &  0 & -2 & $-\overline{H}$ &  $-H$ & -1 & -1 \\
$\chi_{31}$ &      0 &  0 & -1 & -1 &  0 &  0 &  0 &  0 &  0 & -1 & -1 &  0 &  0 &  0 &  0 &  0 &  0 &  0 &  0 &  1 &  1 &  1 &  1 &  1 \\
$\chi_{32}$ &      0 &  0 &  0 &  0 &  0 & -1 &  1 &  1 &  1 &  0 &  0 &  0 &  0 &  0 & -1 &  0 &  0 &  0  & 0 &  0 &  0 &  0 &  0 &  0 \\
$\chi_{33}$ &      1 &  1 &  0 &  0 &  0 &  1 &  0 &  0 &  0 &  0 &  0 &  0 &  0 &  0 & -1 & -1 & -1 &  0 &  0 &  0 &  0 &  0 &  0 &  0 \\
$\chi_{34}$ &      0 &  0 &  1 &  1 &  0 &  0 &  0 & -1 & -1 &  0 &  0 &  0 &  0 &  0 &  0 &  0 &  0 &  1 &  1 & -1 & -1 & -1 &  0 &  0 \\
$\chi_{35}$ &      $C$ & $\overline{C}$ &   0 &  0 &  1 &  0 &  0 &  0 &  0 &  0 &  0 &  0 &  0 &  0 &  0 &  $C$ & $\overline{C}$ &  0 &  0 &  0 &  0 &  0 &  0 &  0 \\
$\chi_{36}$ &     $\overline{C}$ &   $C$ &  0 &  0 &  1 &  0 &  0 &  0 &  0 &  0 &  0 &  0 &  0 &  0 &  0 & $\overline{C}$ &  $C$ &  0 &  0 &  0 &  0 &  0 &  0 &  0 \\
$\chi_{37}$ &      0 &  0 &  3 &  0 &  0 &  0 &  0 &  0 &  0 &  0 &  0 &  1 &  1 &  1 &  0 &  0 &  0 &  0 &  0 & -1 & -1 & -1 & -1 & -1 \\
$\chi_{38}$ &      2 &  2 &  0 &  0 &  0 &  0 &  0 &  0 &  0 &  0 &  0 & -1 & -1 & -1 &  0 &  0 &  0 &  1 &  1 &  0 &  0 &  0 & -1 & -1 \\
$\chi_{39}$ &      0 &  0 & -1 & -1 &  0 &  0 & -1 &  0 &  0 & -1 & -1 &  0 &  0 &  0 &  1 &  0 &  0 &  0 &  0 &  1 &  1 &  1 &  0 &  0 \\
$\chi_{40}$ &     -1 & -1 &  2 & -1 &  0 &  0 &  0 &  0 &  0 &  0 &  0 &  0 &  0 &  0 &  0 &  1 &  1 &  0 &  0 &  0 &  0 &  0 & -1 & -1 \\
$\chi_{41}$ &      0 &  0 &  0 &  0 &  0 &  0 &  0 &  0 &  0 &  1 &  1 &  0 &  0 &  0 & -1 &  0 &  0 &  0 &  0 &  0 &  0 &  0 &  0 &  0 \\
$\chi_{42}$ &      0 &  0 &  0 &  0 &  0 &  1 &  0 &  0 &  0 &  0 &  0 &  0 &  0 &  0 & -1 &  0 &  0 & -1 & -1 &  0 &  0 &  0 &  0 &  0 \\
$\chi_{43}$ &      0 &  0 &  0 &  0 &  1 &  0 &  1 & -1 & -1 &  0 &  0 &  0 &  0 &  0 & -1 &  0 &  0 &  0 &  0 &  0 &  0 &  0 &  1 &  1 \\
$\chi_{44}$ &      0 &  0 & -2 &  1 &  1 &  0 & -1 &  0 &  0 &  0 &  0 &  0 &  0 &  0 &  0 &  0 &  0 &  0 &  0 &  0 &  0 &  0 &  0 &  0 \\
$\chi_{45}$ &      0 &  0 &  0 &  0 &  0 & -1 & -1 &  1 &  1 &  0 &  0 &  0 &  0 &  0 &  1 &  0 &  0 &  0 &  0 &  0 &  0 &  0 & -1 & -1 \\
$\chi_{46}$ &      0 &  0 &  0 &  0 &  0 &  0 &  0 &  0 &  0 &  0 &  0 &  0 &  0 &  0 &  0 &  0 &  0 &  0 &  0 &  0 &  0 &  0 &  0 &  0 \\
$\chi_{47}$ &      0 &  0 &  0 &  0 &  0 &  0 &  1 &  0 &  0 &  0 &  0 &  0 &  0 &  0 & -1 &  0 &  0 &  0 &  0 &  0 &  0 &  0 &  1 &  1 \\
$\chi_{48}$ &      0 &  0 &  0 &  0 & -1 &  0 &  0 &  0 &  0 &  0 &  0 &  0 &  0 &  0 &  1 &  0 &  0 &  0 &  0 &  0 &  0 &  0 &  0 &  0 \\
\hline
\end{tabular}
\caption{Character table of $Th$, Part IV.}
\label{char4}
\end{table}
\end{center}
\end{landscape}
\begin{landscape}
\begin{table}[htp]
\begin{tabular}{ccccccccccc}
\hline
$[g]$            & 1A       &  2A   & 3A   & 3B   & 3C    & 4A     & 4B     & 5A     &  6A     & 6B    \\
\hline 
$v,h$            & 0,1      & 0,1   & 1,3  & 0,1  & 2,3   & 0,1    & 7,8    & 0,1    & 5,6     & 2,3   \\
\hline
$\kappa_{m,g}$   & $240_1$  & 0     & $-6_1+18_9$ & $6_1$ & 0 & $8_4$ & 0 & 0 & 0 & 0 \\
\hline
$N_{[g]}$        & 4        & 8     & 36   & 12   & 36    & 16     & 32     & 20     & 72      & 72 \\
\hline
 & \\
  & \\
\hline
$[g]$            & 6C       &  7A   & 8A   & 8B   & 9A    & 9B     & 9C     & 10A    &  12AB   & 12C   \\
\hline 
$v,h$            & 0,1      & 0,1   & 7,8  & 13,16& 0,1   & 0,1    & 1,3    & 0,1    & 7,12    & 0,1   \\
\hline
$\kappa_{m,g}$   & $0$  & $2_1$     & $0$ & $0$ & $6_9$ & $-3_9$ & 0 & 0 & $-1_4+3_{36}$ & $-1_4$ \\
\hline
$N_{[g]}$        & 24       & 28    & 64   & 128  & 36    & 36     & 108    & 40     & 144     & 48 \\
\hline
 & \\
  & \\
\hline
$[g]$            & 12D      &  13A  & 14A  & 15AB & 18A   & 18B    & 19A    & 20A    &  21A    & 24AB  \\
\hline 
$v,h$            & 19,24    & 0,1   & 0,1  & 1,3  & 0,1   & 2,3    & 0,1    & 7,8    & 1,3    & 19,24  \\
\hline
$\kappa_{m,g}$   & $0$  & $\left(\tfrac 13\right)_1$  & $0$ & $0$ & $0$ & $0$ & $\left(\tfrac 35\right)_1$ & 0 & $\left(\tfrac 18\right)_1-\left(\tfrac 38\right)_{9}$ & $0$ \\
\hline
$N_{[g]}$        & 288      & 52    & 56   & 180  & 72    & 216    & 76    & 160     & 252     & 576 \\
\hline
 & \\
  & \\
\hline
$[g]$            & 24CD     &  27A  & 27BC & 28A  & 30AB  & 31AB   & 36A    & 36BC   &  39AB    &   \\
\hline 
$v,h$            & 37,48    & 1,3   & 1,3  & 0,1  & 2,3   & 0,1    & 0,1    & 0,1    & 1,3    &   \\
\hline
$\kappa_{m,g}$   & $0$  & $-1_9+3_{81}$  & $\left(\frac 12\right)_9-\left(\frac 32\right)_{81}$ & $1_4$ & $0$ & $-\left(\tfrac 14\right)_1$ & $2_4-3_{36}$ & $-1_4$ & $-\left(\frac 37\right)_1+\left(\frac 97\right)_{9}$ &  \\
\hline
$N_{[g]}$        & 1152     & 324   & 324  & 112  & 360    & 124    & 144    & 144     & 468     &  \\
\hline
\end{tabular}
\caption{Multipliers and theta corrections and levels. The notation $(\kappa_{m,g})_m$ indicates the addition of the theta correction $\sum_m\kappa_{m,g}\vartheta(m\tau)$ in the definition of $\calF_{[g]}(\tau)$ in \eqref{calF}, see Table 5 in \cite{Harvey} and also \Cref{remark}.}
\label{mult}
\end{table}
\end{landscape}
\newpage
\renewcommand{\arraystretch}{1.8}
\begin{table}[htp]
\begin{tabular}{|r|l|c|c|}
\hline
$[g]$ &                  & Replicable function & Expression \\
\hline
1A    & $\t^{(0)}_{[g]}$ &   $\frac{1^8}{4^8}$ &  $\frac{248x^4 + 57472x^3 + 3735552x^2 + 79691776x + 536870912}{x^4 + 16x^3}$          \\
      & $\t^{(1)}_{[g]}$ &   $\frac{1^8}{4^8}$ & $x+8-86016x^{-1}-3407872x^{-2}-33554432x^{-3}$           \\
      & $\t_{[g]}$       &   $\frac{1^{24}}{2^{24}}$  & $x+272-2^{15}x^{-1} $          \\
      \hline
2A    & $\t^{(0)}_{[g]}$ &   $\frac{2^{24}}{1^{8}\cdot 4^{16}}$ & $-8 + 256x^{-1}$           \\
      & $\t^{(1)}_{[g]}$ &   $\frac{2^{24}}{1^{8}\cdot 4^{16}}$ & $x-8$           \\
      & $\t_{[g]}$       &   $\frac{1^{24}}{2^{24}}$            & $x+16$         \\
      \hline
3A    & $\t_{[g]}$       &  $\frac{2^2\cdot 9}{1\cdot 18^2}$    & $\frac{x^6 + 9x^5 - 126x^4 + 450x^3 - 675x^2 + 405x}{x^5 - 4x^4 + 9x^3 - 15x^2 + 18x - 9}$           \\  
      \hline
3B    & $\t^{(0)}_{[g]}$ &  $\frac{3^3\cdot 4}{1\cdot 12^3}$    & $\frac{5x^4 + 34x^3 - 240x^2 + 448x - 256}{x^4 - 2x^3}$          \\
      & $\t^{(1)}_{[g]}$ &  $\frac{3^3\cdot 4}{1\cdot 12^3}$    &   $x - 1 + 24x^{-1} + 32x^{-2} - 128x^{-3}$         \\
      & $\t_{[g]}$       &  $\frac{2^3\cdot 3^9}{1^3\cdot 6^9}$ & $x + 2 + 64x^{-1}$           \\ 
      \hline 
3C    & $\t_{[g]}$       &  $\frac{2^2\cdot 9}{1\cdot 18^2}$    & $\frac{x^7 - 9x^6 + 9x^5 + 171x^4 - 837x^3 + 1701x^2 - 1701x + 729}{x^6 - 4x^5 + 9x^4 - 15x^3 + 18x^2 - 9x}$           \\  
            \hline 
4A    & $\t^{(0)}_{[g]}$ &      0              &    8        \\
      & $\t^{(1)}_{[g]}$ & $\frac{1^8}{4^8}$   &    $x+8$    \\
      & $\t_{[g]}$       & $\frac{1^8}{4^8}$   &    $x+16$   \\  
            \hline 
4B    & $\t^{(0)}_{[g]}$ &      0              &     0      \\
      & $\t^{(1)}_{[g]}$ & $\frac{2^{12}}{4^{12}}$  &  $x$          \\
      & $\t_{[g]}$       & $\frac{2^{12}}{4^{12}}$  &  $x$          \\ 
      \hline 
5A    & $\t_{[g]}$       &  $\frac{1^3\cdot 5}{2\cdot 10^3}$    & $\frac{x^3 + 10x^2 + 36x + 80}{x^2 + 9x + 20}$   \\  
            \hline
\end{tabular}
\caption{Relations to replicable functions I. In the table's third column we use the short hand $d^r$ for the expression $\eta(d\tau)^r$. Thus, for example the expression $\frac{1^8}{4^8}$ stands for $\eta(\tau)^8\eta(4\tau)^{-8}$. The rational function in the fourth column means that the replicable function from the third column plugged into it gives the functions indicated in the second column.}
\label{repl}
\end{table}

\begin{table}[htp]
\begin{tabular}{|r|l|c|c|}
\hline
$[g]$ &                  & Replicable function & Expression \\
\hline
6A    & $\t_{[g]}$       &   $\frac{2^{2}\cdot 9}{1\cdot 18^{2}}$  & $\frac{x^4 - 18x^2 + 36x - 27}{x^3 - 3x^2 + 3x}$        \\
      \hline
6B    & $\t_{[g]}$       &   $\frac{2^{2}\cdot 9}{1\cdot 18^{2}}$  & $\frac{x^3 - 6x^2 + 9x}{x^2 - 3x + 3}$         \\
      \hline
6C    & $\t^{(0)}_{[g]}$ &  $\frac{4^4\cdot 6^2}{2^2\cdot 12^4}$    & $\frac{x + 3}{x - 1} $          \\
      & $\t^{(1)}_{[g]}$ &  $\frac{4^4\cdot 6^2}{2^2\cdot 12^4}$    &   $ x$        \\
      & $\t_{[g]}$       &  $\frac{2^3\cdot 3^9}{1^3\cdot 6^9}$ & $x-2$          \\ 
      \hline 
8A    & $\t^{(0)}_{[g]}$ &      0              &    0        \\
      & $\t^{(1)}_{[g]}$ & $\frac{2^4}{8^4}$   &    $x$    \\
      & $\t_{[g]}$       & $\frac{2^4}{8^4}$   &    $x$    \\
            \hline 
8B    & $\t^{(0)}_{[g]}$ &      0              &    0        \\
      & $\t^{(1)}_{[g]}$ & $\frac{4^6}{8^6}$   &    $x$    \\
      & $\t_{[g]}$       & $\frac{4^6}{8^6}$   &    $x$   \\  
            \hline 
9A    & $\t_{[g]}$       & $\frac{2^2\cdot 9}{1\cdot 18^2}$  &  $\frac{x^3 + 3x^2 - 15x + 27}{x^2 - x}$\\ 
      \hline 
9B    & $\t_{[g]}$       &  $\frac{2^2\cdot 9}{1\cdot 18^2}$    & $\frac{x^4 - 6x^3 + 12x^2 - 18x + 27}{x^3 - x^2}$   \\ 
\hline  
10A    & $\t^{(0)}_{[g]}$ &  $\frac{2^4\cdot 10^4}{1^3\cdot 4\cdot 5\cdot 20^3}$ &    $2-4x^{-1}$        \\
      & $\t^{(1)}_{[g]}$ & $\frac{1^2\cdot 10^5}{2\cdot 4\cdot 5^2\cdot 20^3}$   &    $x+2$    \\
      & $\t_{[g]}$       & $\frac{2^4\cdot 5^2}{1^2\cdot 10^4}$ &    $x-5x^{-1}$    \\ 
            \hline
12AB  & $\t^{(0)}_{[g]}$       &  $\frac{1\cdot 4\cdot 18}{2\cdot 9 \cdot 36}$    & $\frac{2x}{x-3}$   \\ 
\hline  
12C    & $\t^{(0)}_{[g]}$ &  0 &    -1        \\
      & $\t^{(1)}_{[g]}$ & $\frac{4^4\cdot 6^2}{2^2\cdot 12^4}$  &    $x$    \\
      & $\t_{[g]}$       & $\frac{4^4\cdot 6^2}{2^2\cdot 12^4}$ &    $x-1$    \\ 
            \hline
12D    & $\t^{(0)}_{[g]}$ &  0 &    0        \\
      & $\t^{(1)}_{[g]}$ & $\frac{6^4}{12^4}$  &    $x$    \\
      & $\t_{[g]}$       & $\frac{6^4}{12^4}$ &    $x$    \\ 
            \hline
\end{tabular}
\caption{Relations to replicable functions II.}
\label{eplii}
\end{table}

\begin{table}[htp]
\begin{tabular}{|r|l|c|c|}
\hline
$[g]$ &                  & Replicable function & Expression \\
\hline
14A    & $\t^{(0)}_{[g]}$ &  $\frac{2^3\cdot 14^3}{1\cdot 4^2\cdot 7\cdot 28^2}$    & $-1+4x^{-1} $          \\
      & $\t^{(1)}_{[g]}$ &  $\frac{2^3\cdot 14^3}{1\cdot 4^2\cdot 7\cdot 28^2}$    &   $ x-1$        \\
      & $\t_{[g]}$       &  $\frac{1^3\cdot 7^3}{2^3\cdot 14^3}$ & $x+2$          \\ 
      \hline 
18A    & $\t^{(0)}_{[g]}$ & $\frac{4\cdot 9}{1\cdot 36}$  &    $\frac{x-1}{x+1}$        \\
      & $\t^{(1)}_{[g]}$ & $\frac{2\cdot 12^4\cdot 18}{4^2\cdot 6^2\cdot 36^2}$   &    $x$    \\
      & $\t_{[g]}$       & $\frac{2^2\cdot 9}{1\cdot 18^2}$   &    $x-3x^{-1}$    \\
            \hline 
20A    & $\t^{(0)}_{[g]}$ &      0              &    0        \\
      & $\t^{(1)}_{[g]}$ & $\frac{2^2\cdot 10^2}{4^2\cdot 20^2}$   &    $x$    \\
      & $\t_{[g]}$       & $\frac{2^2\cdot 10^2}{4^2\cdot 20^2}$   &    $x$   \\  
            \hline 
24AB    & $\t_{[g]}$       & 0  &  0\\ 
      \hline 
24CD   & $\t^{(0)}_{[g]}$ &  0 &    0        \\
      & $\t^{(1)}_{[g]}$ & $\frac{12^2}{24^2}$   &    $x$    \\
      & $\t_{[g]}$       & $\frac{12^2}{24^2}$ &    $x$    \\ 
            \hline
28A  & $\t^{(0)}_{[g]}$       &  0    & 1   \\ 
     & $\t^{(1)}_{[g]}$ & $\frac{1\cdot 7}{4\cdot 28}$   &    $x+1$    \\
      & $\t_{[g]}$       & $\frac{1\cdot 7}{4\cdot 28}$ &    $x+2$    \\ 
\hline  
36A    & $\t^{(0)}_{[g]}$ &  $\frac{2^5\cdot 3\cdot 12\cdot 18}{1^2\cdot 4^2\cdot 6^2\cdot 9\cdot 36}$ &    $-1+6x^{-1}$        \\
      & $\t^{(1)}_{[g]}$ & $\frac{2\cdot 12^4\cdot 18}{4^2\cdot 6^2\cdot 36^2}$  &    $x$    \\
      & $\t_{[g]}$       &  &    $\t^{(0)}_{[g]}+t^{(1)}_{[g]}$    \\ 
            \hline
36BC    & $\t^{(0)}_{[g]}$ &  0 &    -1        \\
      & $\t^{(1)}_{[g]}$ & $\frac{2\cdot 12^4\cdot 18}{4^2\cdot 6^2\cdot 36^2}$  &    $x$    \\
      & $\t_{[g]}$       & $\frac{1\cdot 12\cdot 18^3}{4\cdot 6\cdot 9\cdot 36^2}$ &    $x$    \\ 
            \hline
\end{tabular}
\caption{Relations to replicable functions III.}
\label{repliii}
\end{table}

\section{Congruences}\label{AppCongruences}
Here we give the linear relations and congruences for the McKay--Thompson series of Thompson moonshine. In the relations below the symbol $[g]$ represents the McKay--Thompson series $\mathcal T_{[g]},$ and the symbol $[\vartheta_{m^2}]$ represents $\vartheta(m^2 \tau).$ \\

\subsection{Linear Relations}\label{TblLinearRelations}
\begin{align*}
0&=3[9A]-[3A]-[3C]-[3B]\\
&=[9B]-[9A]+9[\vartheta_9]\\
&=3[18A]-[6A]-[6B]-[6C]\\
&=[27BC]-[27A]-3/2[\vartheta_9]-9/2[\vartheta_{81}]\\
&=[36BC]-[36A]+3[\vartheta_4]+3[\vartheta_{36}]
\end{align*}\\

\subsection{Congruences}\label{secCongruences}\mbox{}\\

{\bf p=31:}
\begin{align*}
0&\equiv [1A]-[31AB] &\pmod{31}
\end{align*}

{\bf p=19:}
\begin{align*}
0&\equiv[1A]-[19A] & \pmod{19}
\end{align*}

{\bf p=13:}
\begin{align*}
0&\equiv[1A] -[13A]& \pmod{13}\\
&\equiv[3A]-[39AB] & \pmod{13}
\end{align*}

{\bf p=7:}
\begin{align*}
0&\equiv[1A]-[7A] && \pmod{7^2}\\
&\equiv[2A]-[14A] && \pmod{7}\\
&\equiv[3A]-[21A] && \pmod{7}\\
&\equiv[4A]-[28A] && \pmod{7}
\end{align*}

{\bf p=5:}
\begin{align*}
0&\equiv [1A]-[5A] &&\pmod{5^2}\\
&\equiv[2A]-[10A] &&\pmod{5}\\
&\equiv[3C]-[15AB] &&\pmod{5}\\
&\equiv[4B]-[20A] &&\pmod{5}\\
&\equiv[6A]-[30AB] &&\pmod{5}
\end{align*}

{\bf p=3:}
\begin{align*}
0&\equiv [1A] - [3A] +9 [\vartheta_9] && \pmod{3^3}\\
&\equiv 28[1A] - 27 [3A] - [3B] && \pmod{3^7}\\
&\equiv 53[1A] -27 [3A] -53[3B] + 27[3C] + 3^6 [\vartheta_9]&& \pmod{3^8}\\
&\equiv [2A] - [6A] && \pmod{3}\\
&\equiv 2[2A] - [6A] - [6B] && \pmod{3^2}\\
&\equiv 2[2A] -3[6A] + [6C] && \pmod{3^3}\\
&\equiv 760[1A] - 864[3A] +212 [3B] +621 [3C] - 3^6 [9C]\\
& \ \ \ \ \ \ \ \ - 49\cdot3^6 [\vartheta_9] - 3^9 [\vartheta_{81}] && \pmod{3^{10}}\\
&\equiv [4A] - [12AB] && \pmod{3}\\
&\equiv 4[4A] - 3[12AB] - [12C] - 9[\vartheta_4] - 18[\vartheta_{36}] && \pmod{3^3}\\
&\equiv [4B] - [12D] && \pmod{3}\\
&\equiv [5A] - [15AB] &&\pmod{3}\\
&\equiv 16[2A] - 15[6A] +8[6C] - 9[18B] && \pmod{3^4}\\
&\equiv [7A] - [21A] && \pmod{3}\\
&\equiv [8A] - [24AB] && \pmod{3}\\
&\equiv [8B] - [24CD] && \pmod{3}\\
&\equiv [3C] +2[9C] - 3[27A] - 21[\vartheta_9] && \pmod{3^3}\\
&\equiv [10A] - [30AB] && \pmod{3}\\
&\equiv [12C] - [36A] && \pmod{3}\\
&\equiv [13A] - [39AB] && \pmod{3}
\end{align*}

{\bf p=2:}
\begin{align*}
0 &\equiv [1A] - [2A] - 2^8 [\vartheta_4] && \pmod{2^{12}}\\
&\equiv [1A] + 15 [2A] - 16 [4A] && \pmod{2^{13}}\\
&\equiv 11[1A]+5[2A]-144[4A]+2^7[4B] -3\cdot2^9[\vartheta_4] && \pmod{2^{16}}\\
&\equiv [3C] - [6A] && \pmod{2^3}\\
&\equiv [3A] - [6B] - 8 [\vartheta_4] - 8 [\vartheta_36] && \pmod{2^5}\\
&\equiv [3B] - [6C] +4[\vartheta_1] && \pmod{2^3}\\
&\equiv [1A] -17 [2A] -16 [4A] + -96 [4B] + 2^7 [8A] -2^8[\vartheta_4] && \pmod{2^{13}}\\
&\equiv 3[1A] + 13 [2A] -112 [4A] + 480 [4B] -7\cdot2^7 [8A] \\
&\ \ \ \ \ \ \ \ + 2^9 [8B] + 2^8 [\vartheta_4] && \pmod{2^{15}}
\end{align*}
\begin{align*}
0&\equiv [5A] - [10A] +4 [\vartheta_1] && \pmod{2^3}\\
&\equiv [3A] -5 [6B] + 4 [12AB] + 80 [\vartheta_4] + 16 [\vartheta_{36}] && \pmod{2^7}\\
&\equiv 5[3B] + [6C] + 2 [12C] + 8 [\vartheta_1] && \pmod{2^4}\\
&\equiv 3 [3C] -5 [6A] + 2 [12D] && \pmod{2^5}\\
&\equiv [7A] - [14A] +4 [\vartheta_1] && \pmod{2^3}\\
&\equiv [9C] - [18B] && \pmod{2^2}\\
&\equiv 3 [5A] + 3 [10A] + 2 [20A] && \pmod{2^4}\\
&\equiv 3 [3A] + [24AB] + 5 [\vartheta_4] + [\vartheta_{36}] && \pmod{2^3}\\
&\equiv [12D] - [24CD] && \pmod{2^3}\\
&\equiv 5 [7A] + [14A] + 2 [28A] && \pmod{2^4}\\
&\equiv [15AB] + [30AB] && \pmod{2^2}\\
&\equiv [3C] + [6A] + 2 [12C] + 2 [24AB] + 2 [36A] + 6 [\vartheta_4] + 8 [\vartheta_9] + 6 [\vartheta_36] && \pmod{2^4}
\end{align*}\\

\noindent Additionally, for each $g$ of odd order we have the congruence
\[0\equiv [g]-\alpha_{[g]}(0)\pmod{2},\] 
and of course 
\[1\equiv [\theta_{m^2}]\pmod{2}.\] 

\end{document}